\documentclass[11pt, oneside]{amsart}

\usepackage{amsmath,amsthm,amssymb,hyperref}
\usepackage{mathrsfs} 
\usepackage{lipsum}

\usepackage{tikz, tikz-cd}
\usepackage[all]{xy}
\usetikzlibrary{arrows,decorations.markings}
\usetikzlibrary{backgrounds,shapes}
\usetikzlibrary{patterns,hobby}
        \usepackage{pgfplots}
        \pgfplotsset{compat=1.6}
\tikzset{>=latex}
\usepackage{subfig}
\usepackage{booktabs}
\usepackage{subfiles}
\usepackage{caption}
\usepackage{makecell}
\usepackage[new]{old-arrows}

\usepackage{color}
\usepackage{hyperref}
\hypersetup{
    colorlinks=true, 
    linktoc=all,     
    linkcolor={mauve},  
    citecolor={plum},
    filecolor={plum},
    urlcolor={plum}
}

\definecolor{smoked}{RGB}{216, 212, 204}
\definecolor{mauve}{RGB}{200, 55, 171}
\definecolor{apricot}{RGB}{250, 144, 4}
\definecolor{sky}{RGB}{66, 169, 244}
\definecolor{plum}{RGB}{76, 0, 102}

\definecolor{lightmauve}{RGB}{232, 173, 220}
\definecolor{lightapricot}{RGB}{253, 211, 155}
\definecolor{lightsky}{RGB}{178, 221, 251}
\definecolor{lightplum}{RGB}{184, 153, 192}

\usepackage{geometry} 
\usepackage[parfill]{parskip}

\newcommand{\Z}{{\mathbb Z}}

\newcommand{\R}{{\mathbb R}}

\newcommand{\Q}{{\mathbb Q}}

\makeatletter
\newcommand{\oset}[3][0ex]{%
  \mathrel{\mathop{#3}\limits^{
    \vbox to#1{\kern-2\ex@
    \hbox{$\scriptstyle#2$}\vss}}}}
\makeatother

\DeclareMathOperator\Out{Out}

\DeclareMathOperator\Hom{Hom}

\newcommand\restr[2]{{%
  \left.\kern-\nulldelimiterspace
  #1
  \vphantom{\big|} 
  \right|_{#2}
  }}

\DeclareMathOperator{\PSL}{\mathrm{PSL}}
\DeclareMathOperator{\psl}{\PSL_2\R}

\theoremstyle{plain}                    
\newtheorem{thm}{Theorem}[section]
\newtheorem{thma}{Theorem}

\newtheorem{lem}[thm]{Lemma}
\newtheorem{prop}[thm]{Proposition}

\newtheorem{fact}[thm]{Fact}	
\newtheorem{assumption}[thm]{Assumption}

\theoremstyle{definition}
\newtheorem{defn}[thm]{Definition}

\theoremstyle{remark}
\newtheorem{rmk}[thm]{Remark}
\newtheorem{claim}[thm]{Claim}

\numberwithin{equation}{section}

\newenvironment{proofclaim}
 {\proof}
 {\endproof}

\newcommand{\surface}{S}
\DeclareMathOperator{\Rep}{Rep}
\newcommand{\RepDT}{\Rep^\mathrm{DT}_{\alpha}(\surface)}
\newcommand{\RepDTarg}[2]{\Rep^\mathrm{DT}_{#1}(#2)}

\DeclareMathOperator{\PMod}{PMod}

\newcommand{\nuG}{\nu_{\mathcal{G}}}

\title[Invariant measures for Deroin--Tholozan representations]{Invariant measures for Deroin--Tholozan representations}

\author{Yohann Bouilly}
\address[Y.~Bouilly]{Université de Strasbourg, 5/7 rue Ren\'e Descartes, 67084 Strasbourg, France.}
\email{bouilly@math.unistra.fr}

\author{Arnaud Maret}
\address[A.~Maret]{Université de Strasbourg, IRMA, 7 rue Descartes, 67000 Strasbourg, France.}
\email{maret.arnaud@unistra.fr}

\date{\today}

\begin{document}

\begin{abstract}
    We classify mapping class group invariant probability measures on the character varieties of Deroin--Tholozan representations, namely the compact components of relative $\mathrm{PSL}_2\mathbb{R}$-character varieties. We prove that an ergodic measure is either the counting measure on a finite orbit or agrees with the Liouville measure induced by the Goldman symplectic form. Our approach is based on measure disintegration along transverse Lagrangian tori fibrations.
\end{abstract}

\maketitle

\section{Introduction}
\subsection{Motivations and result}
Mapping class group dynamics on character varieties of surface group representations come in different flavors depending on the nature of the representations. For instance, the action is properly discontinuous on components of Fuchsian representations, or more generally of Anosov representations. In contrast, this work focuses on the cases where the dynamics exhibit chaotic behavior. More precisely, we study the dynamics on character varieties of \emph{totally elliptic} surface group representations into $\psl$, also known as \emph{Deroin--Tholozan} (DT) representations (Definition~\ref{def:DT-representations}). These character varieties, which we call \emph{DT components}, are compact, even though totally elliptic representations typically have dense image in $\psl$. A theorem of Mondello shows that DT components actually account for all non-trivial compact connected components of relative $\psl$-character varieties
~\cite{mondello}, where ``relative'' means that conjugacy classes have been prescribed for the images of peripheral curves, highlighting the pivotal role of total ellipticity to characterize compactness. In fact, non-trivial compact components exist only when the underlying surface is a sphere with at least four punctures and the peripheral monodromies are all elliptic.

Our work continues a line of research on mapping class group action on DT components. This action preserves a natural volume measure known as the \emph{Goldman measure}. The action was proved to be ergodic in~\cite{ergodicity}, and its closed invariant subsets have been classified in~\cite{orbit-closures} as either finite orbits or entire DT components; the classification of finite orbits was completed in~\cite{finite-orbits}. Our main result provides a classification of all mapping class group invariant measures on DT components.

\begin{thma}\label{main-thm} 
Let $\mu$ be a Borel probability measure on a DT component, ergodic with respect to the mapping class group action. Then $\mu$ is either the counting measure on a finite orbit or the Goldman measure, i.e.~the Liouville measure associated with the Goldman symplectic form.
\end{thma}

Theorem~\ref{main-thm} has been previously obtained in the special case of $4$-punctured spheres by Cantat--Dupont--Martin-Baillon~\cite[Corollary~3.3]{stationary-measures}. It also answers a question raised by the authors and Faraco in~\cite{orbit-closures}. To the best of our knowledge, classifications of mapping class group invariant measures on character varieties are scarce, aside from Theorem~\ref{main-thm} and the result of~\cite{stationary-measures} for 4-punctured spheres. In particular, the problem appears to remain open even for $\mathrm{SU}(2)$ character varieties.

As established in~\cite{finite-orbits}, the existence of finite mapping class group orbits inside DT components is an extremely rare phenomenon. In particular, all orbits are infinite if the underlying sphere has at least seven punctures. On all DT components with no finite orbits, Theorem~\ref{main-thm} asserts unique ergodicity of the mapping class group action. One may thus wonder whether Theorem~\ref{main-thm} recovers the minimality results of~\cite[Theorem~A]{orbit-closures}. Although unique ergodicity together with full support implies minimality for actions of locally compact amenable groups, the mapping class groups considered here are not amenable, and we are therefore not aware of any such implication in this setting. 

Related works include a recent contribution by Cantat--Dupont--Martin-Baillon that classifies closed invariant subspaces and invariant measures for the action of $\mathrm{Aut}(F_n)$ by pre-composition on $\Hom(F_n,G)$, for compact Lie groups $G$ and sufficiently large $n$
~\cite{AutFn}. Here, $F_n$ denotes the free group on $n$ generators. Although $F_n$ is isomorphic to the fundamental group of an $(n+1)$-punctured sphere $\surface$, only very specific automorphisms of $F_n$ are induced by homeomorphisms of $\surface$. Our result is thus of a different nature.

A natural continuation of this work is the classification of \emph{stationary measures}, which are, by definition, invariant only in an averaged sense (see e.g.~\cite[Definition~8.23]{einsiedler-ward}). For 4-punctured spheres, it was shown in~\cite{stationary-measures} that the only ergodic stationary measures on DT components are in fact invariant; consequently, they are either counting measures along finite orbits or the Goldman measure. For spheres with more punctures, however, the classification of ergodic stationary measures on DT components remains open. 

\subsection{Strategy of the proof}
Our strategy to prove Theorem~\ref{main-thm} is inspired from~\cite{stationary-measures}. After excluding the case where the measure $\mu$ has atoms, we proceed as follows.

DT components are symplectic toric manifolds and therefore naturally fiber into isotropic tori, most of which are Lagrangian. The general idea is to \emph{disintegrate} $\mu$ along this fibration (Section~\ref{sec:disintegration}). The conditional measures on the fibers are invariant under the action of a non-trivial subgroup of the mapping class group that preserves the fibration. The residual dynamics on the fibers is sufficiently rich---it contains irrational rotations of most Lagrangian tori---to force almost all the conditional measures to be Lebesgue measures. We apply this argument twice at every point in a DT component, disintegrating $\mu$ along two Lagrangian tori fibrations that are transverse at that point. In this way, we conclude that the localization of the measure $\mu$ around the point is absolutely continuous with respect to the Goldman measure. Since this holds at every point, it follows that $\mu$ coincides with the Goldman measure.

There are two technical results to prove along the way:
\begin{enumerate}
    \item First, assuming that $\mu$ has no atoms, we show that the measure of every isotropic torus is zero (Proposition~\ref{prop:pushforward-no-atoms}). We do so by constructing explicit mapping class group elements for which the iterated images of a given fiber are ``almost'' pairwise disjoint.
    \item Second, we prove the required transversality result: for every point in a DT component, there exist two toric fibrations such that the Langrangian fibers through the point are transverse (Proposition~\ref{prop:transversality}). 
\end{enumerate}

For the proofs of Propositions~\ref{prop:pushforward-no-atoms} and~\ref{prop:transversality}, we rely on the interpretation of DT representations as \emph{triangle chains} in the hyperbolic plane (Section~\ref{sec:triangle-chains}). This allows for explicit computations using hyperbolic trigonometry.

We emphasize that our proof of Theorem~\ref{main-thm} does not rely on the classification of closed invariant subsets established in~\cite{orbit-closures} (see also Remark~\ref{rem:independence-from-minimality}). In particular, we do not use the fact that $\mu$ has full support when it is atom-free.

\subsection{Organization of the paper}
Section~\ref{sec:preliminaries} reviews key properties of DT representations and their description in terms of triangle chains. It also includes a brief recap on measure disintegration.

We prove that Lagrangian fibers have zero measure (Proposition~\ref{prop:pushforward-no-atoms}) in Section~\ref{sec:no-atoms}. Some of the more involved trigonometric computations are postponed to Appendix~\ref{apx:trig-computations}. The transversality statement on Lagrangian tori arising from distinct fibrations (Proposition~\ref{prop:transversality}) is addressed in Section~\ref{sec:transversality}. 

Finally, in Section~\ref{sec:proof}, we complete the proof of Theorem~\ref{main-thm} by detailing the disintegration argument and invoking the unique ergodicity of irrational torus rotations.

\subsection{Acknowledgements}
We are grateful to Nicolas Tholozan for initially suggesting the problem of classifying invariant measures on DT components. Thanks to Pierre-Louis Blayac and Florestan Martin-Baillon for inspiring discussions. AM was supported by the European Research Council (ERC) under the European Union's Horizon 2020 research and innovation program (Grant agreement No. 101096550).

\section{Preliminaries}\label{sec:preliminaries}
\subsection{DT representations}
We briefly recap the definition of DT representations and state some of their important properties. For more details, the reader is invited to consult~\cite[Section~2]{orbit-closures} or~\cite[Section~2.2]{finite-orbits}.

\subsubsection{Definition}
Let $\surface$ denote an oriented sphere with a finite set $\mathcal{P}$ of $n\geq 4$ punctures. Counterclockwise loops around a puncture of $\surface$ are called \emph{peripheral loops}. Among all representations $\pi_1\surface\to\psl$, we are interested in those mapping all peripheral loops on $\surface$ to non-trivial elliptic elements of $\psl$. More precisely, given an angle vector $\alpha\in (0,2\pi)^\mathcal{P}$, we denote by $\Rep_\alpha(\surface,\psl)$ the set of conjugacy classes of representations $\pi_1\surface\to\psl$ that map all peripheral loops around a puncture $p\in\mathcal{P}$ to elements of $\psl$ with counterclockwise \emph{rotation angle} $\alpha_p$, i.e.~ conjugates of
\[
\pm\begin{pmatrix}
    \cos(\alpha_p/2) & \sin(\alpha_p/2)\\
    -\sin(\alpha_p/2) & \cos(\alpha_p/2)
\end{pmatrix}.
\]
The space $\Rep_\alpha(\surface,\psl)$ is called \emph{$\alpha$-relative character variety}.

\begin{defn}\label{def:DT-representations}
    A representation $\rho\colon\pi_1\surface\to\psl$ whose conjugacy class $[\rho]$ belongs to $\Rep_\alpha(\surface,\psl)$ is a \emph{DT representation} if
    \begin{itemize}
        \item $\rho(\pi_1\surface)$ is a Zariski dense subgroup of $\psl$,\footnote{This condition can be relaxed to the requirement that $\rho(\pi_1\surface)$ is not contained in any conjugate of $\mathrm{PSO}(2)$, see~\cite[Theorem~A]{totally-elliptic}. The condition may even be dropped if $\alpha$ satisfies~\eqref{eq:angle-condition}, since all representations with peripheral monodromy $\alpha$ are automatically Zariski dense when $\sum_{p\in\mathcal{P}}\alpha_{p}\notin 2\pi\Z$.} and
        \item $\rho$ is \emph{totally elliptic}, meaning that every element of $\pi_1\surface$ that represents a simple closed curve on $\surface$ is mapped to an elliptic element of $\psl$.
    \end{itemize}
\end{defn}

DT representations were first discovered by Benedetto--Goldman when $n=4$~\cite{benedetto-goldman}, and for all $n\geq 4$ by Deroin--Tholozan~\cite{deroin-tholozan}. Their characterization in terms of total ellipticity was established in~\cite{totally-elliptic}. It turns out that DT representations exist if and only if
\begin{equation}\label{eq:angle-condition}
\sum_{p\in\mathcal{P}} \alpha_p<2\pi \quad \text{or}\quad \sum_{p\in\mathcal{P}} \alpha_p>2\pi(n-1).
\end{equation}
When $\alpha$ satisfies~\eqref{eq:angle-condition}, Deroin--Tholozan proved that the subset of conjugacy classes of DT representations form a smooth connected component of the $\alpha$-relative character variety called \emph{DT component}. We denote it by
\[
\RepDT\subset\Rep_\alpha(\surface,\psl).
\]
They also proved that every DT component is compact and symplectomorphic to $\mathbb{CP}^{n-3}$, where $n$ is the number of punctures on the underlying sphere~\cite[Theorem~4]{deroin-tholozan}. 

\subsubsection{Symplectic toric structure}
The symplectic structure on $\RepDT$ is given by the Goldman symplectic form~\cite{goldman-symplectic}, which we denote by $\omega_\mathcal{G}$. Two related notions will be relevant for this paper:
\begin{itemize}
    \item The Poisson bracket defined by $\{f,g\}=\omega_\mathcal{G}(X_f,X_g)$, where $X_f, X_g$ are the Hamiltonian vector fields of two smooth functions $f,g\colon\RepDT\to\R$.
    \item The Liouville measure $\nuG$ associated to $\omega_{\mathcal{G}}$ which we normalize to a probability measure on $\RepDT$ and call the \emph{Goldman measure}.
\end{itemize}

Every DT component carries the structure of a symplectic toric manifold once a \emph{pants decomposition} $\mathcal{B}$ of $\surface$ is fixed, i.e.\ a maximal collection of disjoint, non-peripheral simple closed curves $b_1,\ldots,b_{n-3}$ on $S$. Since DT representations are totally elliptic by Definition~\ref{def:DT-representations}, the pants decomposition $\mathcal{B}$ determines a moment map $\beta\colon \RepDT \to (0,2\pi)^{n-3}$ by sending $[\rho]$ to the rotation angles of $\rho(b_1),\ldots,\rho(b_{n-3})$. Its image $\Delta_\mathcal{B}\subset \mathbb{R}^{n-3}$ is an equivalent copy of the standard simplex known as moment polytope~\cite[Lemma~3.5]{deroin-tholozan}. The fibers of $\beta$ over $\Delta_\mathcal{B}$ define a fibration of $\RepDT$ into isotropic tori. The top dimensional fibers are Lagrangian tori and coincide with the fibers over the interior $\oset{\circ}{\Delta}_\mathcal{B}$ of $\Delta_\mathcal{B}$.

\subsubsection{Dynamics}
The action on $\RepDT$ that we study is that of the \emph{pure mapping class group} of $\surface$, denoted by $\PMod(\surface)$ and defined as the group of orientation-preserving homeomorphisms of $\surface$ that fix each puncture individually, up to isotopy. By the Dehn--Nielsen--Baer theorem (see e.g.~\cite[Theorem~8.8]{mcg-primer}), this group identifies with a subgroup of $\Out(\pi_1\surface)$, and thus acts on $\RepDT$ by pre-composition.

The action was proved to be ergodic in~\cite{ergodicity}. Furthermore, closed invariant subsets have been classified in~\cite{orbit-closures}: they are either finite orbits (independently classified in~\cite{finite-orbits}), or the whole DT component $\RepDT$.

\subsubsection{Triangle chains}\label{sec:triangle-chains}
There is a simple geometric model for DT representations developed in~\cite{action-angle} in terms of chains of hyperbolic triangles in the upper half-plane. The construction depends on two pieces of data:
\begin{enumerate}
\item First, a \emph{chained pants decomposition} $\mathcal{B}$ of $\surface$, i.e.~a pants decomposition in which each pair of pants contains at least one puncture of $\surface$.
\item Second, a presentation of $\pi_1\surface$ with generators $c_1,\ldots,c_n$ satisfying $c_1\cdots c_n=1$ (we say that the presentation is \emph{geometric}) which is \emph{compatible} with $\mathcal{B}$, meaning that the pants curves are represented by the fundamental group elements $b_i=(c_1\cdots c_{i+1})^{-1}$ for $i=1,\ldots,n-3$, see Figure~\ref{fig:standard-pants-decomposition}.
\end{enumerate} 

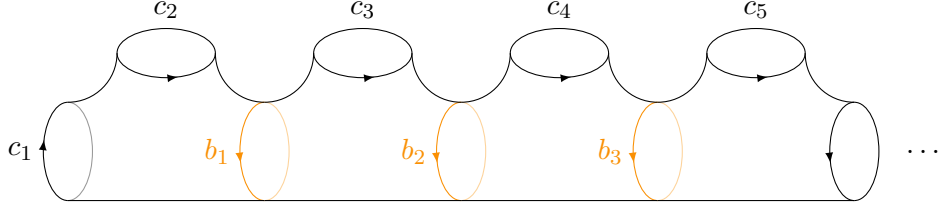
\begin{figure}[h]
    \centering
\begin{tikzpicture}[scale=1.3, decoration={
    markings,
    mark=at position 0.6 with {\arrow{>}}}]
  \draw[postaction={decorate}] (0,-.5) arc(-90:-270: .25 and .5) node[midway, left]{$c_1$};
  \draw[black!40] (0,.5) arc(90:-90: .25 and .5);
  \draw[apricot, postaction={decorate}] (2,.5) arc(90:270: .25 and .5) node[midway, left]{$b_1$};
  \draw[lightapricot] (2,.5) arc(90:-90: .25 and .5);
  \draw[apricot, postaction={decorate}] (4,.5) arc(90:270: .25 and .5) node[midway, left]{$b_2$};
  \draw[lightapricot] (4,.5) arc(90:-90: .25 and .5);
  \draw[apricot, postaction={decorate}] (6,.5) arc(90:270: .25 and .5) node[midway, left]{$b_3$};
  \draw[lightapricot] (6,.5) arc(90:-90: .25 and .5);
  \draw[postaction={decorate}] (8,.5) arc(90:270: .25 and .5);
  \draw (8,.5) arc(90:-90: .25 and .5);
  
  \draw (.5,1) arc(180:0: .5 and .25) node[midway, above]{$c_2$};
  \draw[postaction={decorate}] (.5,1) arc(-180:0: .5 and .25);
  \draw (2.5,1) arc(180:0: .5 and .25)node[midway, above]{$c_3$};
  \draw[postaction={decorate}] (2.5,1) arc(-180:0: .5 and .25);
  \draw (4.5,1) arc(180:0: .5 and .25)node[midway, above]{$c_4$};
  \draw[postaction={decorate}] (4.5,1) arc(-180:0: .5 and .25);
  \draw (6.5,1) arc(180:0: .5 and .25)node[midway, above]{$c_5$};
  \draw[postaction={decorate}] (6.5,1) arc(-180:0: .5 and .25);
   
  \draw (0,.5) to[out=0,in=-90] (.5,1);
  \draw (1.5,1) to[out=-90,in=180] (2,.5);
  \draw (0,-.5) to[out=0,in=180] (2,-.5);
  
  \draw (2,.5) to[out=0,in=-90] (2.5,1);
  \draw (3.5,1) to[out=-90,in=180] (4,.5);
  \draw (2,-.5) to[out=0,in=180] (4,-.5);
  
  \draw (4,.5) to[out=0,in=-90] (4.5,1);
  \draw (5.5,1) to[out=-90,in=180] (6,.5);
  \draw (4,-.5) to[out=0,in=180] (6,-.5);
  
  \draw (6,.5) to[out=0,in=-90] (6.5,1);
  \draw (7.5,1) to[out=-90,in=180] (8,.5);
  \draw (6,-.5) to[out=0,in=180] (8,-.5);
  
  \draw (8.7, 0) node{$\ldots$};
\end{tikzpicture}
    \caption{The pants curves $b_i$ in a compatible system of generators of $\pi_1\surface$.}
    \label{fig:standard-pants-decomposition}
\end{figure}

Given a DT representation $\rho\colon\pi_1\surface\to\psl$, we denote by $C_1,\ldots,C_n$ and $B_1,\ldots,B_{n-3}$ the unique fixed points of $\rho(c_1),\ldots,\rho(c_n)$, respectively $\rho(b_1),\ldots,\rho(b_{n-3})$, in the upper half-plane. The \emph{$\mathcal{B}$-triangle chain} of $\rho$ is the union of the $n-2$ following triangles:
\[
(C_1,C_2,B_1), (B_1,C_3,B_2),\ldots,(B_{n-4},C_{n-3},B_{n-3}), (B_{n-3},C_{n-1},C_n).
\]
An example of triangle chain is illustrated on Figure~\ref{fig:triangle-chain}. The vertices $C_1,\ldots,C_{n-3}$ are the \emph{exterior} vertices, while $B_1,\ldots,B_{n-3}$ are the \emph{shared} vertices. A triangle chain is called \emph{regular} if all $n-2$ triangles are non-degenerate. Since each triangle corresponds to a pair of pants on $\surface$, a triangle may only be degenerate when all three vertices coincide.
\begin{figure}[h]
\centering
\begin{tikzpicture}[font=\sffamily, scale=1.2, decoration={markings, mark=at position 1.01 with {\arrow{>}}}]
\node[anchor=south west,inner sep=0] at (0,0) {\includegraphics[width=10.8cm]{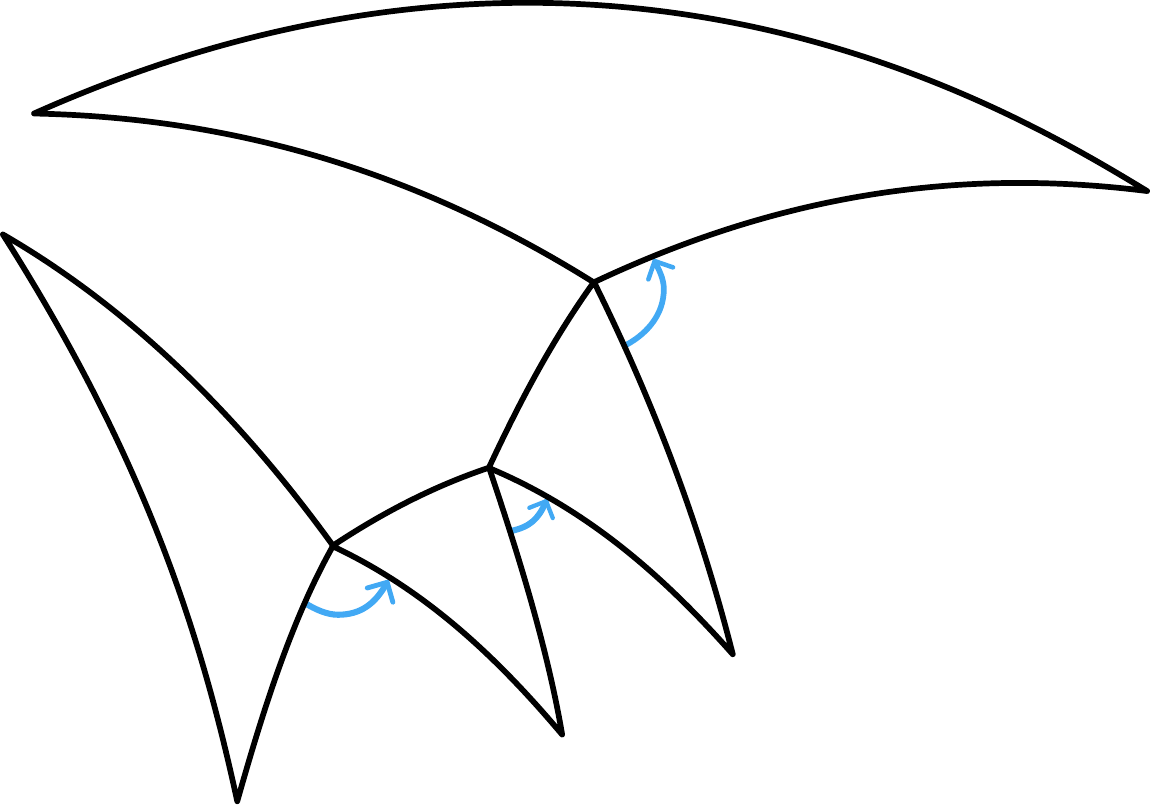}};

\begin{scope}
\fill (0.27,5.4) circle (0.07) node[left]{$C_1$};
\fill (9.01,4.81) circle (0.07) node[right]{$C_2$};
\fill (5.75,1.2) circle (0.07) node[right]{$C_3$};
\fill (4.4,.56) circle (0.07) node[below]{$C_4$};
\fill (1.85,0.03) circle (0.07) node[below]{$C_5$};
\fill (0.03,4.46) circle (0.07) node[left]{$C_6$};
\end{scope}

\begin{scope}[apricot]
\fill (4.63,4.06) circle (0.07) node[left]{$B_1$};
\fill (3.83,2.65) circle (0.07) node[above left]{$B_2$};
\fill (2.64,2.05) circle (0.07);
\draw (2.65,2.12) node[above]{$B_3$};
\end{scope}

\begin{scope}[sky]
\draw (5.4,3.8) node{$\gamma_1$};
\draw (4.3,2.05) node{$\gamma_2$};
\draw (2.8,1.3) node{$\gamma_3$};
\end{scope}
\end{tikzpicture}
\caption{A regular triangle chain corresponding to a DT representation of a $6$-punctured sphere.}
\label{fig:triangle-chain}
\end{figure}
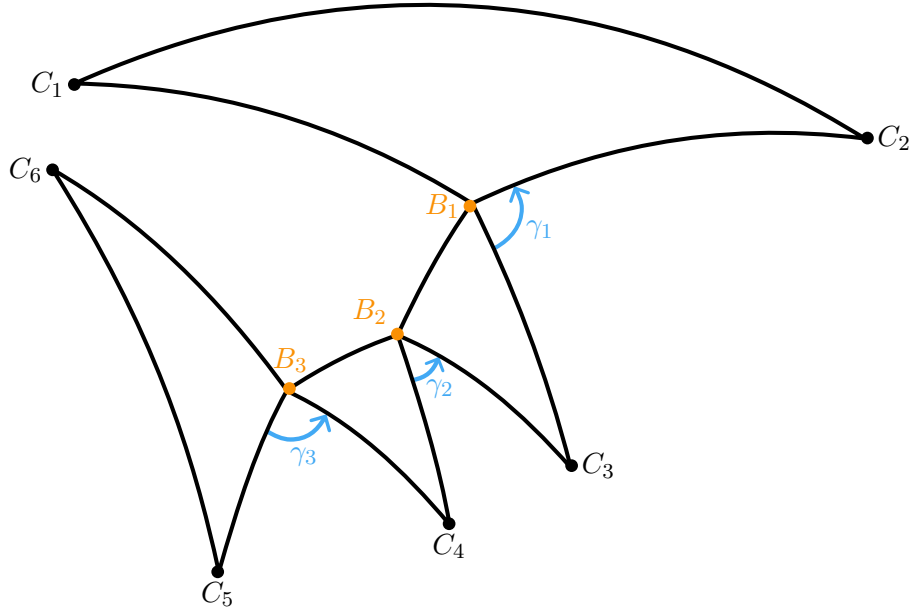

\begin{fact}[{\cite[Lemma~3.8]{action-angle}}]\label{fact:regular-triangle-chains}
The $\mathcal{B}$-triangle chain of $\rho$ is regular if and only if $[\rho]$ belongs to a top dimensional fiber of $\beta$, i.e.~$\beta([\rho])\in\oset{\circ}{\Delta}_\mathcal{B}$, or equivalently $(d\beta_i)_{[\rho]}\neq 0$ for every $i=1,\ldots,n-3$.
\end{fact}

The next result characterizes the fixed points of certain Dehn twists in terms of $\mathcal{B}$-triangle chains.
\begin{fact}[{\cite[Lemma~2.17]{finite-orbits}}]\label{fact:fixed-points-Dehn-twists}
    If $\tau\in\PMod(\surface)$ denotes the Dehn twist along the simple closed curve represented by the fundamental group element $c_ic_j$ for some $i<j$, then $[\rho]\in\RepDT$ is fixed by $\tau$ if and only if $C_i=\cdots=C_j$ or $C_{j+1}=\cdots=C_n=C_1=\cdots=C_{i-1}$.
\end{fact}

The oriented angles $\gamma_i=\angle(C_{i+2},B_i,C_{i+1})$, illustrated on Figure~\ref{fig:triangle-chain}, ``between" two consecutive non-degenerate triangles in a $\mathcal{B}$-triangle chain are a possible choice of angle coordinates that parametrize the fibers of the moment map $\beta$. In particular, the action-angle coordinates $(\beta_1,\ldots,\beta_{n-3},\gamma_1,\ldots,\gamma_{n-3})$ identify $\beta^{-1}(\oset{\circ}{\Delta}_\mathcal{B})$ diffeomorphically with the product $\oset{\circ}{\Delta}_\mathcal{B}\times (\R/2\pi\Z)^{n-3}$.

\begin{fact}[{see e.g.~\cite[Lemma~2.15]{finite-orbits}}]\label{fact:action-of-Dehn-twist-angle-coordinates}
    The Dehn twists $\tau_1,\ldots,\tau_{n-3}\in\PMod(\surface)$ along the pants curves of $\mathcal{B}$ preserve the fibers of $\beta$ and act on angle coordinates by
    \[
    \gamma_i \circ \tau_j =\begin{cases}
        \gamma_i+\beta_i & \text{if $j=i$,}\\
        \gamma_i & \text{if $j\neq i$.}
    \end{cases}
    \]
    Geometrically, $\tau_i$ acts on $\mathcal{B}$-triangle chains by rotating the sub-chain made of the triangles $(B_i,C_{i+2},B_{i+1}),\ldots,(B_{n-3},C_{n-1},C_n)$ counterclockwise around $B_i$ by an angle $\beta_i$.
\end{fact}

The chosen presentation of $\pi_1\surface$ gives a bijection $\mathcal{P}\leftrightarrow \{1,\ldots,n\}$. The interior angles in the $\mathcal{B}$-triangle chain of $[\rho]\in\RepDT$ can be expressed in terms of the angle vector $\alpha\in (0,2\pi)^n$ and the action coordinates $\beta_i$.
\begin{fact}[{\cite[Lemma~3.5]{action-angle}}]\label{fact:interior-angles-triangle-chains}
   Assume that we are in the regime $\sum_i\alpha_i > 2\pi(n-1)$ of~\eqref{eq:angle-condition}. With the conventions $B_0 = C_1$, $B_{n-2} = C_n$, $\beta_0 = 2\pi - \alpha_1$, and $\beta_{n-2} = \alpha_n$, the triangle $(B_i, C_{i+1}, B_{i+1})$ has interior angles $(\beta_i/2,\, \pi - \alpha_i/2,\, \pi - \beta_{i+1}/2)$.
\end{fact}

\subsection{Measure disintegration}\label{sec:disintegration}
Let $(X, m)$ be a Borel probability space. Given a surjective measurable map $f\colon X\to Y$ to a Borel mesurable space $Y$, we equip $Y$ with the pushforward probability measure $f_\ast m$. The \emph{disintegration} of $m$ along $f$ is given by a system of probability measures $\{m_y\}_{y\in Y}$ on $X$ such that
\[
m=\int_Y m_y \,d(f_\ast m)(y).
\]
The measures $m_y$ are called \emph{conditional measures}. They have the following properties:
\begin{enumerate}
    \item (Support) The measures $m_y$ are supported on the fiber $f^{-1}(y)$ for $f_\ast m$-almost every $y\in Y$.
    \item (Invariance) If $m$ and $f$ are invariant under a group action on $X$, then the measures $m_y$ are also invariant for $f_\ast m$-almost every $y\in Y$.
    \item (Uniqueness) The measures $m_y$ are uniquely determined for $f_\ast m$-almost every $y\in Y$.
\end{enumerate}

For more details and for a proof of the disintegration theorem, the reader may consult~\cite[Theorem~5.14]{einsiedler-ward}.
\section{No atoms in the base}\label{sec:no-atoms}
\subsection{Overview}
The goal of this section is to prove the following result.
\begin{prop}\label{prop:pushforward-no-atoms}
If $\mu$ is an ergodic, atom-free Borel probability measure on $\RepDT$ and $\beta\colon \RepDT\to \Delta_\mathcal{B}$ is the moment map associated to a chained pants decomposition $\mathcal{B}$ of $\surface$, then $\beta_\ast\mu$ has no atoms.
\end{prop}
The proof proceeds by induction on the number of punctures of $\surface$. The strategy is outlined in Section~\ref{sec:strategy-proof-fibers-zero-measure}, and the induction step is carried out in Section~\ref{sec:induction-step-fibers-zero-measure}. Some of the more involved computations are deferred to Appendix~\ref{apx:trig-computations}.
\begin{rmk}
We are assuming that $\mathcal{B}$ is chained in the statement of Proposition~\ref{prop:pushforward-no-atoms}; this is because our proof uses the triangle chains model for DT representations which has only been formally developed for chained pants decompositions. Nevertheless, we expect the statement of Proposition~\ref{prop:pushforward-no-atoms} to hold for all pants decompositions.
\end{rmk}

\subsection{Strategy of the proof}\label{sec:strategy-proof-fibers-zero-measure}
Let us pick an arbitrary point $(\beta_1,\ldots,\beta_{n-3})\in\Delta_\mathcal{B}$ and denote by $T$ the fiber over $(\beta_1\ldots,\beta_{n-3})$. Our goal is to prove that $\mu(T)=0$. To do so, we will use the following elementary result.
\begin{claim}\label{claim:iteration-mapping-class-implies-zero-measure}
If $A\subset \RepDT$ is a measurable set and if there exists a mapping class $f\in\PMod(\surface)$ such that $\mu\big(A\cap f^{m}(A)\big)=0$ for every $m\geq 1$, then $\mu(A)=0$.
\end{claim}
\begin{proofclaim}
By invariance of $\mu$, for every $m>\ell\geq 0$, we have
\[
\mu\big(f^\ell(A)\cap f^m(A)\big)=\mu\big(A\cap f^{m-\ell}(A)\big)=0.
\]
Therefore,
\[
\mu\left(\bigcup_{m\geq 0} f^m(A)\right)=\sum_{m\geq 0}\mu\big(f^m(A)\big)=\sum_{m\geq 0}\mu(A),
\]
from which we deduce that $\mu(A)=0$ because $\mu$ is a probability measure.
\end{proofclaim}

Since $T$ is a fiber of the moment map $\beta$, it is diffeomorphic to a torus $(\R/2\pi\Z)^{k-1}$ for some $1\leq k\leq n-2$. When $k=1$, then $T$ is a point and $\mu(T)=0$ because $\mu$ is atom-free. From now on, let us assume that $2\leq k\leq n-2$. Recall that $k$ also corresponds to the number of non-degenerate triangles in the $\mathcal{B}$-triangle chain of any point in $T$. The $k-1$ angles ``between'' these $k$ non-degenerate triangles are angle coordinates $(\gamma_1,\ldots,\gamma_{k-1})$ which parametrize $T$ as the $(k-1)$-dimensional torus $(\R/2\pi\Z)^{k-1}$.

We will use the following notation. If $0\leq \ell\leq k-1$ is an integer and $\Theta=(\theta_1,\ldots,\theta_\ell)\in (\R/2\pi\Z)^\ell$ is an angle vector (take $\Theta=\emptyset$ when $\ell=0$), then we define the sub-torus $T_\Theta\subset T$ by
\[
T_\Theta=T\cap \bigcap_{i=1}^\ell \left\{\gamma_i=\theta_i\right\}.
\]
In other words, $T_\Theta$ is the set of points in $T$ for which the values of the angle coordinates $\gamma_1,\ldots,\gamma_\ell$ have been prescribed. We said that $T_\Theta$ is a sub-torus of $T$ because it can be parametrized by $\gamma_{\ell+1},\ldots, \gamma_{k-1}$ as the $(k-\ell-1)$-dimensional torus $(\R/2\pi\Z)^{k-\ell-1}$. In particular, when $\ell=0$, then $T_\emptyset = T$ and when $\ell=k-1$, then $T_\Theta$ is a point and $\mu(T_\Theta)=0$.

We will prove that $\mu(T)=0$ using a decreasing induction on the value of $\ell$, starting with the base case $\ell=k-1$. To prove the induction step, we will assume that $\mu(T_\Theta)=0$ for every angle vector $\Theta$ of length at least $\ell_0$ for some $1\leq \ell_0\leq k-1$. This will be our inductive hypothesis.
\begin{assumption}[Inductive hypothesis]\label{ass:inductive-hypothesis-fibers-zero-measure}
The equality $\mu(T_\Theta)=0$ holds for every angle vector $\Theta$ of length at least $\ell_0$.
\end{assumption}

\subsection{The induction step}\label{sec:induction-step-fibers-zero-measure}
We fix a system of geometric generators $c_1,\ldots,c_n$ of $\pi_1\surface$ compatible with the pants decomposition $\mathcal{B}$.

Let now $\Theta\in (\R/2\pi\Z)^{\ell_0 -1}$ be an angle vector of length $\ell_0-1$. 
From the definition of $T_\Theta$, the values of $\gamma_1,\ldots, \gamma_{\ell_0-1}$ are prescribed for every point in $T_\Theta$. The next angle coordinate $\gamma_{\ell_0}$ measures the angle between two consecutive non-degenerate triangles in the $\mathcal{B}$-triangle chain of any point in $T_\Theta$. The first of the two triangles has vertices $(B_{i-2},C_i,B_{i-1})$ and the second has vertices $(B_{j-2},C_j,B_{j-1})$, with the convention that $B_0=C_1$ when $i=2$ and $B_{n-2}=C_n$ when $j=n-1$ (see Figure~\ref{fig:triangle-chain-2}). When $T$ is a regular fiber of $\beta$ (which is the case when $(\beta_1,\ldots,\beta_{n-3})$ lies in $\oset{\circ}{\Delta}_\mathcal{B}$), then $j=i+1$. However, when $T$ is a singular fiber, then it is possible that $j>i+1$. In that case, the vertices $B_{i-1},B_i,\ldots, B_{j-2}$ and $C_{i+1},C_{i+2},\ldots, C_{j-1}$ all coincide.

\begin{figure}[h]
\centering
\begin{tikzpicture}[font=\sffamily, scale=1.1]
\node[anchor=south west,inner sep=0] at (0,0) {\includegraphics[width=10.8cm]{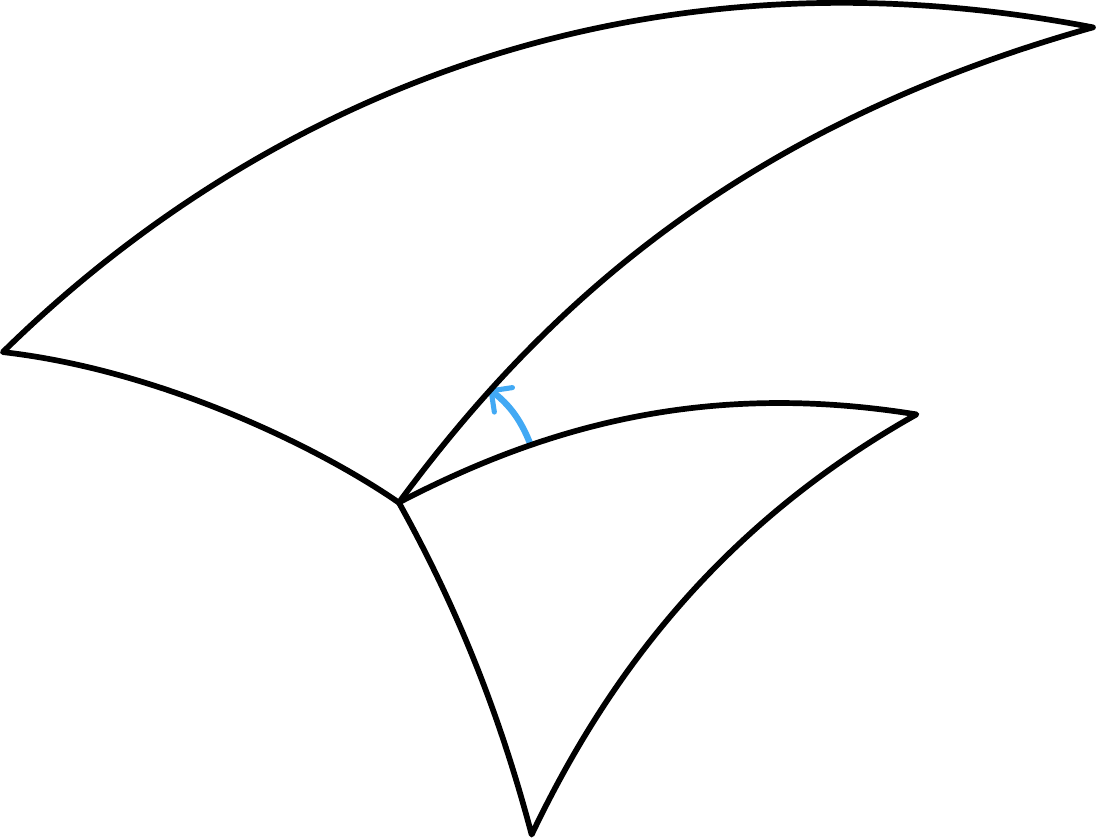}};

\begin{scope}
\fill (9.8,7.25) circle (0.07) node[right]{$C_i$};
\fill (8.24,3.78) circle (0.07) node[right]{$C_j$};
\end{scope}

\begin{scope}[apricot]
\fill (0.06,4.36) circle (0.07) node[left]{$B_{i-2}$};
\fill (3.6,3) circle (0.07) node[left]{$B_{i-1}=\cdots=B_{j-2}$};
\fill (4.77,0.04) circle (0.07) node[below]{$B_{j-1}$};
\draw (3.6,2.6) node[left]{$=C_{i+1}=\cdots=C_{j-1}$};
\end{scope}

\begin{scope}[sky]
\draw (5,3.9) node{$\gamma_{\ell_0}$};
\end{scope}
\end{tikzpicture}
\caption{The consecutive triangles $(B_{i-2},C_i,B_{i-1})$ and $(B_{j-2},C_j,B_{j-1})$.}
\label{fig:triangle-chain-2}
\end{figure}
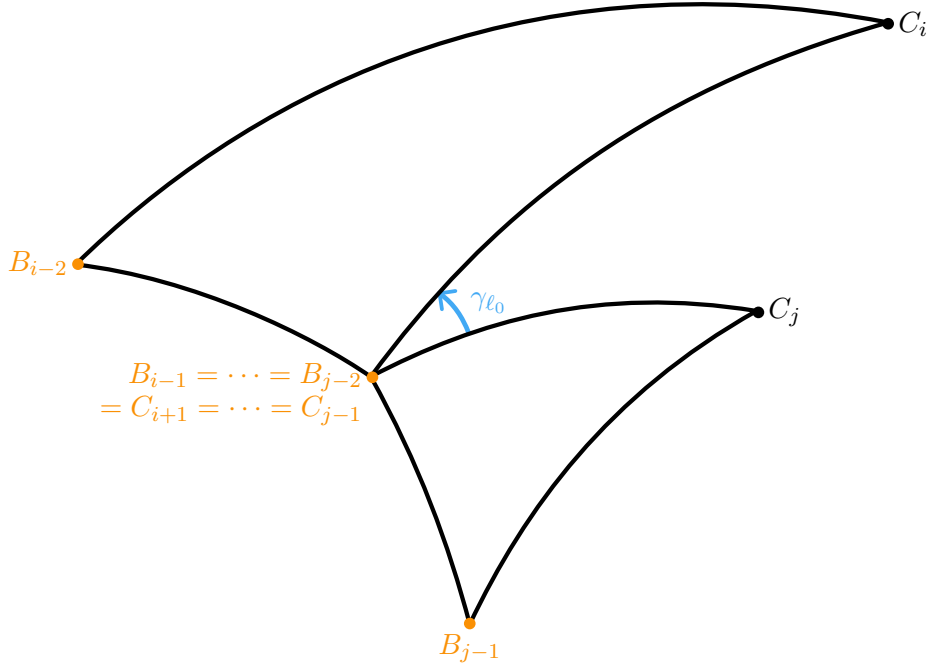

We will consider the Dehn twist $\tau\in\PMod(\surface)$ along the simple closed curve represented by the fundamental group element $c_ic_j\in\pi_1\surface$. The Dehn twist $\tau$ will play the role of the mapping class $f$ from Claim~\ref{claim:iteration-mapping-class-implies-zero-measure}.
\begin{lem}\label{lem:induction-step-fibers-zero-measure}
There exists a countable set $\mathcal{S}\subset \R/2\pi \Z$ such that for every $m\geq 1$,
\[
T_\Theta\cap \tau^{-m}(T_\Theta)\subset \bigcup_{s\in\mathcal{S}} T_{(\Theta\vert s)},
\]
where $(\Theta\vert s)$ denotes the angle vector in $(\R/2\pi \Z)^{\ell_0}$ obtained by appending $s$ after the last component of $\Theta$.
\end{lem}

An immediate consequence of Lemma~\ref{lem:induction-step-fibers-zero-measure} is that $\mu\big(T_\Theta\cap \tau^{-m}(T_\Theta)\big)=0$ for every $m\geq 1$, because $\mu(T_{(\Theta\vert s)})=0$ for every $s\in\mathcal{S}$ by the inductive hypothesis (Assumption~\ref{ass:inductive-hypothesis-fibers-zero-measure}). So, using Claim~\ref{claim:iteration-mapping-class-implies-zero-measure}, we conclude $\mu(T_\Theta)=0$. This proves the induction step. In order to complete the proof of Proposition~\ref{prop:pushforward-no-atoms}, it thus remains to prove Lemma~\ref{lem:induction-step-fibers-zero-measure}.

\begin{proof}[Proof of Lemma~\ref{lem:induction-step-fibers-zero-measure}]
Fix an integer $m\geq 1$ and pick any point $[\rho]\in T_\Theta\cap \tau^{-m}(T_\Theta)$. Since $[\rho]$ and $\tau^m[\rho]$ both belong to $T_\Theta\subset T$, it holds $\beta\big([\rho]\big)=\beta\big(\tau^m[\rho]\big)=(\beta_1,\ldots,\beta_{n-3})$ and $\gamma_i\big([\rho]\big)=\gamma_i\big(\tau^m [\rho]\big)$ for $i=1,\ldots, \ell_0-1$. So, a priori, $[\rho]$ and $\tau^m[\rho]$ only differ by the values of the angle coordinates $\gamma_{\ell_0},\ldots,\gamma_{k-1}$.

\begin{claim}\label{claim:comparaison-B-triangle-chains}
Actually, $[\rho]$ and $\tau^m[\rho]$ only differ by the value of $\gamma_{\ell_0}$. In other words, $[\rho]$ and $\tau^m[\rho]$ have the same angle coordinates $\gamma_{\ell_0+1},\ldots,\gamma_{k-1}$. 
\end{claim}
\begin{proofclaim}
As usual, we will denote the vertices of the $\mathcal{B}$-triangle chain of $[\rho]$ by $C_1,\ldots, C_n$ and $B_1,\ldots, B_{n-3}$. In order to better understand the $\mathcal{B}$-triangle chain of $\tau^m[\rho]$, we consider a new geometric presentation of $\pi_1\surface$ given by the fundamental group elements 
\[
(c_i, c_j, c_{j+1},\ldots,c_n,c_1,\ldots,c_{i-1}, \zeta^{-1} c_{i+1}\zeta,\ldots, \zeta^{-1} c_{j-1}\zeta),
\]
where $\zeta=c_jc_{j+1}\cdots c_nc_1\cdots c_{i-1}$. This presentation is compatible with a pants decomposition $\Upsilon$ of $\surface$. The first pants curve in $\Upsilon$ is represented by the fundamental group element $c_ic_j$ which is precisely the curve along which $\tau$ is the Dehn twist. Since $\rho$ is a DT representation, it is totally elliptic by Definition~\ref{def:DT-representations} and we denote by $\upsilon_1\in (0,2\pi)$ the angle of rotation of $\rho(c_ic_j)$. The first triangle in the $\Upsilon$-triangle chain of $\rho$ has vertices $(C_i,C_j,Y_1)$ (it is the mauve triangle on Figure~\ref{fig:triangle-chain-gamma=pi}). If it is non-degenerate, then it has interior angles $(\pi-\alpha_i/2, \pi-\alpha_j/2,\pi-\upsilon_1/2)$ (Fact~\ref{fact:interior-angles-triangle-chains}). The other exterior vertices of the $\Upsilon$-triangle chain of $\rho$ are $C_{j+1},\ldots,C_n,C_1,\ldots,C_{i-1}$ and $\rho(\zeta)^{-1}C_{i+1},\ldots, \rho(\zeta)^{-1}C_{j-1}$. 

The action of $\tau^{m}$ on $[\rho]$ can easily be described in terms of the $\Upsilon$-triangle chain of $\rho$: the exterior vertices $C_i$ and $C_j$ are rotated counterclockwise around $Y_1$ by an angle $m\upsilon_1$ and the exterior vertices $C_{j+1},\ldots,C_n,C_1,\ldots,C_{i-1}$ are not moved (Fact~\ref{fact:action-of-Dehn-twist-angle-coordinates}). Note, that the exterior vertices $\rho(\zeta)^{-1}C_{i+1},\ldots, \rho(\zeta)^{-1}C_{j-1}$ may be moved by $\tau^{m}$ because the word $\zeta$ contains the generator $c_j$, but this will be irrelevant for our argument.\footnote{To be more precise, here is an explicit automorphism of $\pi_1\surface$ that represents $\tau$: 
\[
c_k\mapsto \begin{cases}
    c_k & \text{if $k=1,\ldots,i-1,j+1,\ldots,n$},\\
    (c_j^{-1}c_i^{-1})c_k(c_ic_j) & \text{if $k=i,j$},\\
    (c_j^{-1}c_i^{-1}c_jc_i)c_k(c_i^{-1}c_j^{-1}c_ic_j) & \text{if $k=i+1,\ldots,j-1$}.
\end{cases}
\]} Since the exterior vertices $C_{j+1},\ldots,C_n$ are also exterior vertices in the $\mathcal{B}$-triangle chain of $\rho$ and they are unaffected by $\tau^m$, the shared vertices $B_{j-1},\ldots,B_{n-3}$ are also fixed. We conclude that the angle coordinates $\gamma_{\ell_0+1},\ldots, \gamma_{k-1}$ of $[\rho]$ and $\tau^{m}[\rho]$ coincide. 
\end{proofclaim}

Now that we have a better understanding of the $\mathcal{B}$-triangle chain of $\tau^m[\rho]$, we can prove that the angle coordinate $\gamma_{\ell_0}\big([\rho]\big)$ can only take countably many values in order for $\tau^m[\rho]$ to also belong to $T_\Theta$ .

\begin{claim}\label{claim:exclude-0-pi}
If $\gamma_{\ell_0}\big([\rho]\big)\neq 0$, then $C_i\neq C_j$ and the triangle $(C_i,C_j,Y_1)$ in the $\Upsilon$-triangle chain of $[\rho]$ is non-degenerate. If furthermore $\gamma_{\ell_0}\big([\rho]\big)\neq \pi$, then $Y_1\neq B_{j-1}$ and $\tau[\rho]\neq [\rho]$.
\end{claim}
\begin{proofclaim}
The first assertion is immediate. Now, assume that $Y_1=B_{j-1}$. Since the triangles $(C_i,C_j,Y_1)$ and $(B_{j-2},C_j,B_{j-1})$ both have an angle $\pi-\alpha_j/2$ at $C_j$, it follows that $C_i$ lies on the geodesic ray from $C_j$ through $B_{i-1}=B_{j-2}$. We are assuming that $C_i\neq B_{i-1}$ and $\gamma_{\ell_0}\big([\rho]\big)\neq 0$, so we must have $\gamma_{\ell_0}\big([\rho]\big)=\pi$, see Figure~\ref{fig:triangle-chain-gamma=pi}. 

\begin{figure}[h]
\centering
\begin{tikzpicture}[font=\sffamily, scale=1.1]
\node[anchor=south west,inner sep=0] at (0,0) {\includegraphics[width=10.8cm]{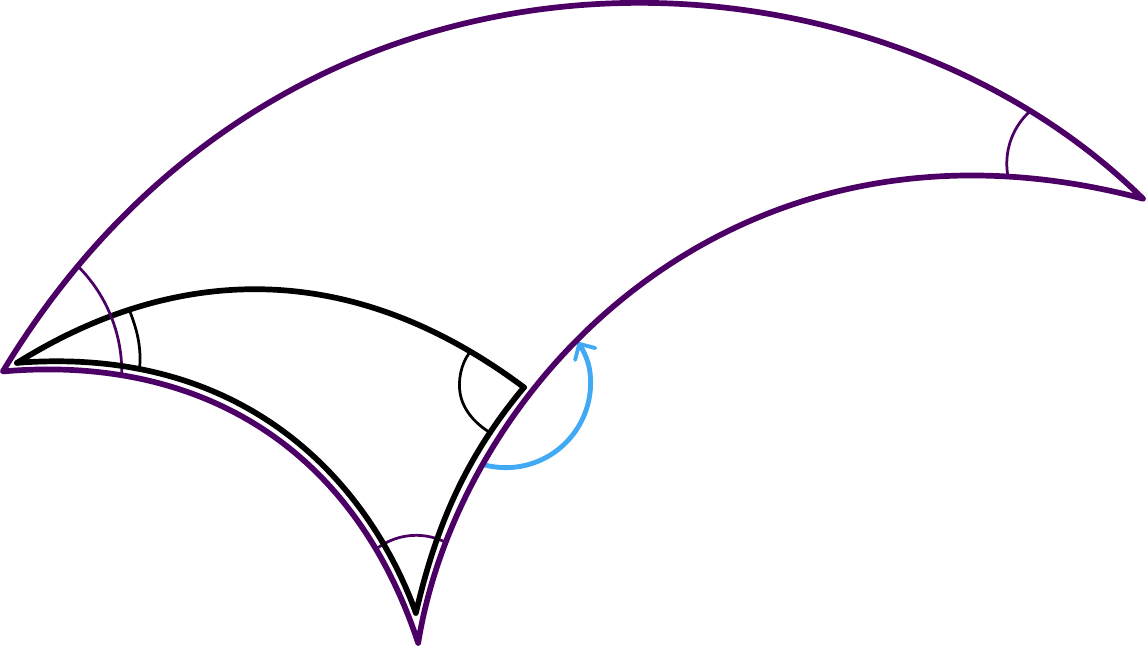}};

\begin{scope}
\fill (9.76,3.85) circle (0.09) node[right]{$C_i$};
\fill (3.59,0.04) circle (0.09) node[below]{$C_j$};
\draw (1.9, 2.63) node{\small $\pi-\frac{\beta_{j-1}}{2}$};
\draw (3.55, 2.2) node{\small $\frac{\beta_{j-2}}{2}$};
\end{scope}

\begin{scope}[plum]
\draw (4.53, .85) node{\small $\pi-\alpha_{j}/2$};
\draw (-.05, 3.35) node{\small $\pi-\upsilon_{1}/2$}; 
\draw (7.92, 4.38) node{\small $\pi-\alpha_{i}/2$};
\end{scope}

\begin{scope}[apricot]
\fill (0.09,2.37) circle (0.09) node[left]{$Y_1=B_{j-1}$};
\fill (4.52,2.2) circle (0.09);
\draw[anchor=south] (4.23, 2.8) node{\small $B_{i-1}=B_{j-2}$};
\draw[->] (4.23, 2.95) to[out=280, in=90] (4.52,2.35);
\end{scope}

\begin{scope}[sky]
\draw (5,2) node[right]{$\gamma_{\ell_0}\big([\rho]\big)=\pi$};
\end{scope}
\end{tikzpicture}
\caption{The configuration in which $Y_1 = B_{j-1}$, forcing $C_i$, $B_{i-1}=B_{j-2}$, and $C_j$ to be collinear. The triangles are drawn slightly offset to highlight their superposition.}
\label{fig:triangle-chain-gamma=pi}
\end{figure}
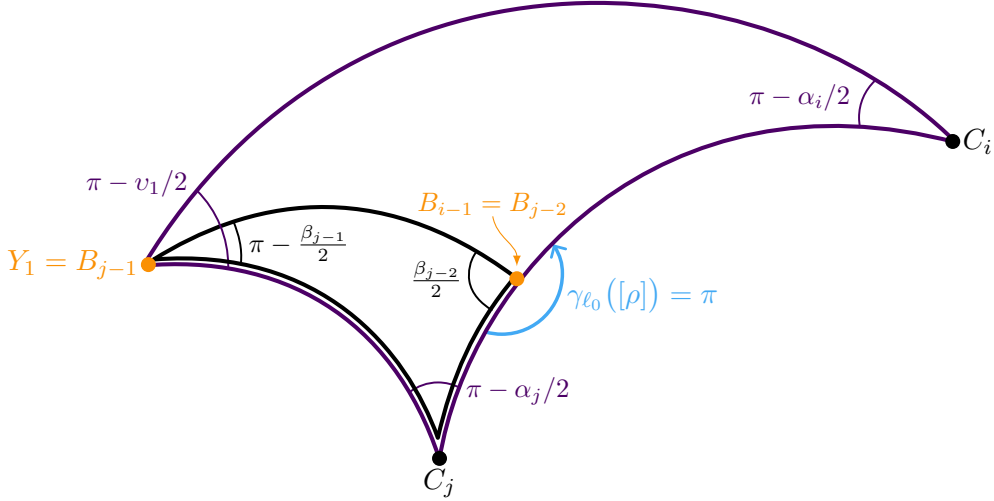

Furthermore, if $\gamma_{\ell_0}\big([\rho]\big)\neq 0$ and $\tau[\rho]=[\rho]$, then all the exterior vertices of the $\Upsilon$-triangle chain of $[\rho]$ except $C_i$ and $C_j$ must coincide with $Y_1$ (Fact~\ref{fact:fixed-points-Dehn-twists}). In particular, this would mean $Y_1=C_{j+1}=\cdots=C_n$ which implies $Y_1=B_{j-1}$ and thus $\gamma_{\ell_0}\big([\rho]\big)= \pi$ by the above.
\end{proofclaim}

\begin{claim}\label{claim:exclude-rationals}
There is a countable subset $\mathcal{Q}\subset \R/2\pi\Z$ containing $0$ and $\pi$ such that if $\gamma_{\ell_0}\big([\rho]\big)\notin \mathcal{Q}$, then $\upsilon_1\notin \pi\Q$ and $\tau^m[\rho]\neq [\rho]$. Moreover, the set $\mathcal{Q}$ does not depend on $m$ and $[\rho]$---it satisfies the conclusion of the claim for any other choice of $m\geq 1$ and $[\rho]\in T_\Theta\cap \tau^{-m}(T_\Theta)$.
\end{claim}
\begin{proofclaim}
If $\gamma_{\ell_0}\big([\rho]\big)\neq 0$, the triangle $(C_i,C_j,Y_1)$ is non-degenerate by Claim~\ref{claim:exclude-0-pi}, see Figure~\ref{fig:triangle-chain-B-and-Y}. Using hyperbolic trigonometry as in~\cite[Equation~(3.9)]{orbit-closures}, we can express $\cos(\upsilon_1/2)$ in terms of $\gamma_{\ell_0}\big([\rho]\big)$ as
\begin{equation}\label{eq:relation-between-upsilon-and-gamma}
\cos(\upsilon_1/2)=\cos\big(\gamma_{\ell_0}([\rho])\big)\cdot k_1 + k_2,
\end{equation}
where $k_1$ and $k_2$ depend only on the values of $\alpha_1,\ldots,\alpha_n$ and $\beta_1,\ldots,\beta_{n-3}$, and $k_1\neq 0$. In particular, to every value of $\upsilon_1\in (0,2\pi)$ corresponds at most two values of $\gamma_{\ell_0}\big([\rho]\big)\in \R/2\pi\Z$ in order for~\eqref{eq:relation-between-upsilon-and-gamma} to hold. So, by choosing $\mathcal{Q}$ to be $\{0\}$ union the set of all angles in $\R/2\pi\Z$ whose cosine can be written as $(\cos(q\pi)-k_2)/k_1$ for some $q\in \Q$, and by taking $\gamma_{l_0}([\rho])\notin\mathcal{Q}$, we can guarantee that $\upsilon_1\notin \pi\Q$.

Now, if we further add $\pi$ to $\mathcal{Q}$, then $\gamma_{\ell_0}\big([\rho]\big)\notin \mathcal{Q}$ also implies $\tau[\rho]\neq [\rho]$ by Claim~\ref{claim:exclude-0-pi}. In order to prove that $\tau^m[\rho]\neq [\rho]$, recall that when $\tau^m$ acts on the $\Upsilon$-triangle chain of $[\rho]$, it rotates the triangle $(C_i,C_j,Y_1)$ around $Y_1$ by an angle $m\upsilon_1$ (Fact~\ref{fact:action-of-Dehn-twist-angle-coordinates}). Under the assumption that $\tau[\rho]\neq [\rho]$, this rotation is trivial if and only if $m\upsilon_1\in 2\pi\Z$, which does not happen as long as  $\gamma_{\ell_0}\big([\rho]\big)\notin \mathcal{Q}$.
\end{proofclaim}

\begin{figure}[h]
\centering
\begin{tikzpicture}[font=\sffamily, scale=1.1]
\node[anchor=south west,inner sep=0] at (0,0) {\includegraphics[width=10.8cm]{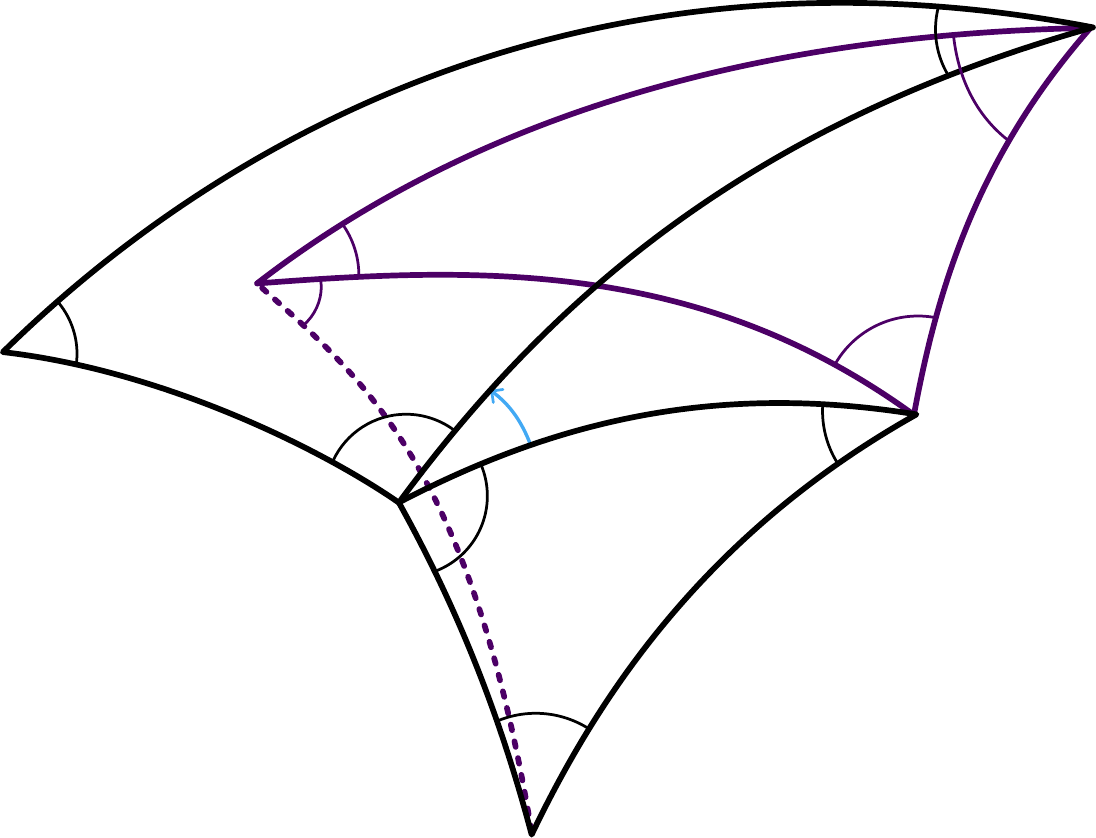}};
\begin{scope}
\fill (9.8,7.25) circle (0.07) node[right]{$C_i$};
\fill (8.2,3.78) circle (0.08) node[right]{$C_j$};
\draw (8.5, 7.7) node{\small $\pi-\alpha_i/2$};
\draw (6.7, 3.5) node{\small $\pi-\alpha_j/2$};
\draw (6.1, 0.85) node{\small $\pi-\beta_{j-1}/2$};
\draw (4.85, 2.8) node{\small $\beta_{j-2}/2$};
\draw (2.2, 3.2) node{\small $\pi-\beta_{i-1}/2$};
\draw (1.25, 4.55) node{\small $\beta_{i-2}/2$};
\end{scope}

\begin{scope}[apricot]
\fill (0.06,4.36) circle (0.07) node[left]{$B_{i-2}$};
\fill (3.6,3) circle (0.07) node[below left]{$B_{i-1}=\cdots=B_{j-2}$};
\fill (4.77,0.04) circle (0.07) node[below]{$B_{j-1}$};
\end{scope}

\begin{scope}[sky]
\draw (5.2,4) node{$\gamma_{\ell_0}\big([\rho]\big)$};
\end{scope}

\begin{scope}[plum]
\fill (2.32,4.97) circle (0.07) node[left]{$Y_1$};   
\draw (2.8, 4.75) node[right]{$\eta$};
\draw (9.75, 6.25) node{\small $\pi-\alpha_i/2$};
\draw (7.65, 4.85) node{\small $\pi-\alpha_j/2$};
\draw (3.85, 5.35) node{\small $\pi-\upsilon_1/2$};
\end{scope}
\end{tikzpicture}
\caption{The configuration in which the triangle $(C_i,C_j,Y_1)$ is non-degenerate and $Y_1\neq B_{j-1}$.}
\label{fig:triangle-chain-B-and-Y}
\end{figure}
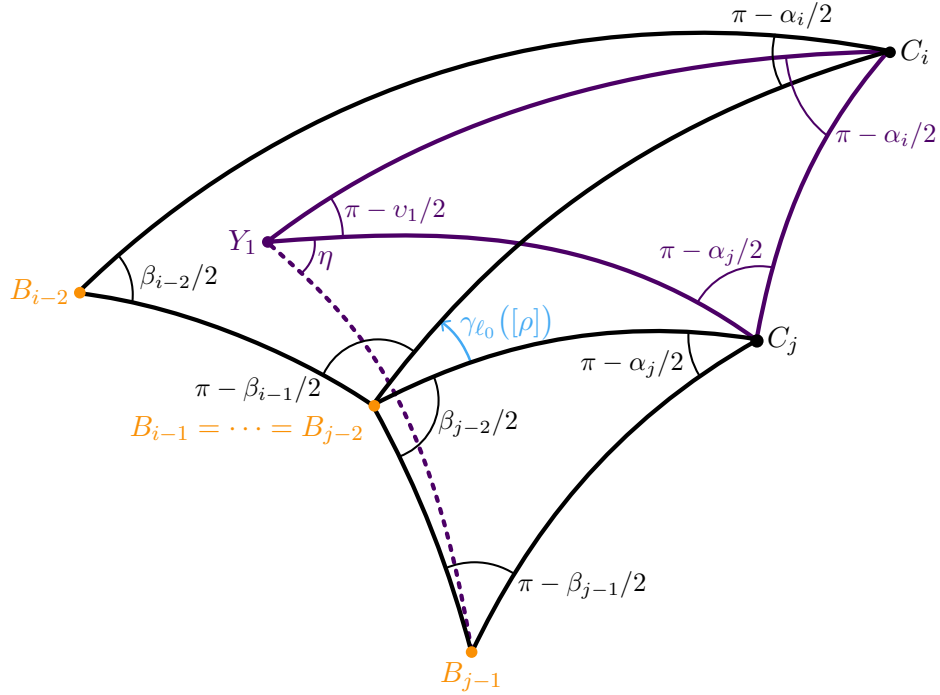

From now on, let us assume that  $\gamma_{\ell_0}\big([\rho]\big)\notin \mathcal{Q}$, where $\mathcal{Q}$ is provided by Claim~\ref{claim:exclude-rationals}. Claim~\ref{claim:exclude-0-pi} implies that $Y_1\neq C_j$ and $Y_1\neq B_{j-1}$, so the oriented angle $\angle (C_j,Y_1,B_{j-1})$ is well-defined, as on Figure~\ref{fig:triangle-chain-B-and-Y}. After applying $\tau^m$ to $[\rho]$, the angle $\angle (C_j,Y_1,B_{j-1})$ becomes $\angle (C_j',Y_1,B_{j-1})$, because both $Y_1$ and $B_{j-1}$ are not moved by $\tau^m$ (Claim~\ref{claim:comparaison-B-triangle-chains}). The precise relation (modulo $2\pi$) between these two angles is given by Fact~\ref{fact:action-of-Dehn-twist-angle-coordinates} and reads
\begin{equation}\label{eq:rotated-angles-fibres-zero-measure}
\angle (C_j',Y_1,B_{j-1})=\angle (C_j,Y_1,B_{j-1})+m\upsilon_1.
\end{equation}
Now, observe that the hyperbolic distances $d(C_j,B_{j-1})$ and $d(C_j',B_{j-1})$ are equal because the triangles $(B_{j-2},C_j,B_{j-1})$ and $(B_{j-2}',C_j',B_{j-1})$ are isometric by Claim~\ref{claim:comparaison-B-triangle-chains} and the discussion beforehand. Hence, $d(C_j,B_{j-1})=d(C_j',B_{j-1})$. Furthermore, $d(C_j,Y_1)=d(C_j',Y_1)$ because $\tau^m$ rotates the triangle $(C_i,C_j,Y_1)$ around $Y_1$. This means that the triangles $(C_j,Y_1,B_{j-1})$ and $(C_j',Y_1,B_{j-1})$ are isometric and hence
\[
\angle (C_j',Y_1,B_{j-1})=\pm\angle (C_j,Y_1,B_{j-1}).
\]
The $\pm$ comes from the fact that the angles are oriented. In the $+$ case, we would get $m\upsilon_1\in 2\pi\Z$ from~\eqref{eq:rotated-angles-fibres-zero-measure} which is impossible by Claim~\ref{claim:exclude-rationals} as we are assuming $\gamma_{\ell_0}\big([\rho]\big)\notin \mathcal{Q}$. So, we conclude from~\eqref{eq:rotated-angles-fibres-zero-measure} that
\begin{equation}\label{eq:eta-equals-mupsilon/2}
\angle (C_j,Y_1,B_{j-1})\in\left\{-\frac{m\upsilon_1}{2}, -\frac{m\upsilon_1}{2}+\pi\right\}.
\end{equation}
For simplicity, we will write $\eta=\angle (C_j,Y_1,B_{j-1})$ (see Figure~\ref{fig:triangle-chain-B-and-Y}). Equation~\eqref{eq:eta-equals-mupsilon/2} overdetermines the polygon with vertices $C_i, C_j, B_{j-1}, Y_1$ and should only be true for finitely many value of $\upsilon_1$. This is what we will prove now.

Equation~\eqref{eq:eta-equals-mupsilon/2} implies that $\cos(\eta)=\pm\cos(m\upsilon_1/2)$, which can further be expressed as a degree $m$ polynomial in $\cos(\upsilon_1/2)$ using Chebyshev polynomials of the first kind. On the other hand, we can use hyperbolic trigonometry in the polygon $C_i, C_j, B_{j-1}, Y_1$ to compute $\cos(\eta)$ in terms of $\cos(\upsilon_1/2)$. By comparing the two relations, we observe that $\cos(\upsilon_1/2)$ is the zero of a degree $2m+4$ polynomial whose coefficients only depend on the angles $\alpha_1,\ldots,\alpha_n$ and $\beta_1,\ldots,\beta_{n-3}$ (hence not on $[\rho]$). The exact formulas and detailed computations can be found in Appendix~\ref{apx:trig-computations}. We conclude that $\cos(\upsilon_1/2)$ can only take finitely many different values in order for~\eqref{eq:eta-equals-mupsilon/2} to hold. To each such value correspond at most two values of $\gamma_{\ell_0}\big([\rho]\big)$ by~\eqref{eq:relation-between-upsilon-and-gamma}. All these values of $\gamma_{\ell_0}\big([\rho]\big)$ form a finite set $\mathcal{Z}_m\subset \R/2\pi\Z$.

In conclusion, by defining $\mathcal{S}=\mathcal{Q}\cup\bigcup_{m\geq 1}\mathcal{Z}_m$, we obtain a countable set such that for every $m\geq 1$ and every $[\rho]\in T_\Theta$, if $\gamma_{\ell_0}\big([\rho]\big)\notin \mathcal{S}$, then $\tau^m[\rho]$ does not belong to $T_\Theta$. This finishes the proof of Lemma~\ref{lem:induction-step-fibers-zero-measure}.
\end{proof}
\section{Transverse Lagrangian tori}\label{sec:transversality}
\subsection{Overview}
The goal of this section is to show that each point in a DT component lies on two transverse Lagrangian torus fibrations arising from chained pants decompositions. More precisely, we prove:
\begin{prop}\label{prop:transversality}
For every $[\rho] \in \RepDT$, there exist chained pants decompositions $\mathcal{B}$ and $\mathcal{D}$ of $\surface$, with associated moment maps $\beta\colon \RepDT\to\Delta_{\mathcal{B}}$ and $\delta\colon \RepDT\to\Delta_{\mathcal{D}}$, such that
\[
T^\ast_{[\rho]}\RepDT = \left\langle (d\beta_i)_{[\rho]}, (d\delta_i)_{[\rho]} : i=1,\ldots, n-3 \right\rangle.
\]
\end{prop}
After preliminary remarks on transversality in character varieties (Section~\ref{sec:preliminary-remarks}), we outline the strategy for proving Proposition~\ref{prop:transversality} (Section~\ref{sec:setting-up-proof}), before constructing the required pants decompositions (Section~\ref{sec:iterated-construction}).

\subsection{Preliminary remarks}\label{sec:preliminary-remarks}
Transversality of Hamiltonian flows associated with simple closed curves on $\surface$ plays a central role in the study of ergodic dynamics on character varieties (see e.g.~\cite[Theorem~2.1 \& Lemma~3.1]{goldman-xia}). In the setting of DT representations, several criteria for transversality have been established: some of an algebraic nature, as in~\cite[Lemmas~5.5 \&~5.6]{ergodicity}, and others formulated in terms of angle coordinates and triangle chains~\cite[Lemmas~3.3 \&~3.10]{orbit-closures}. Proposition~\ref{prop:transversality} differs from these results in that it does not assert transversality for an arbitrary collection of $2(n-3)$ simple closed curves on $\surface$, but rather for a collection that decomposes into two pants decompositions of~$\surface$.

Also note that two Lagrangian tori fibrations of a DT component given by two different pants decompositions of $\surface$ generically intersect transversally. This can be seen by adapting Detcherry's argument to the case of DT representation~\cite[Proposition~4.1]{detcherry}. In contrast, Proposition~\ref{prop:transversality} requires transversality to hold at given points.

\subsection{Setting up the proof}\label{sec:setting-up-proof}
Let us fix an arbitrary $[\rho]\in\RepDT$. It was proved in~\cite[Proposition~2]{aaron-arnaud} that there exists a chained pants decomposition $\mathcal{B}$ of $\surface$ with associated moment map $\beta\colon \RepDT\to\Delta_{\mathcal{B}}$ such that $(d\beta_i)_{[\rho]}\neq 0$ for all $i=1,\ldots,n-3$. In other words, $[\rho]$ belongs to one of the regular fibers of $\beta$ and its $\mathcal{B}$-triangle chain is regular (Fact~\ref{fact:regular-triangle-chains}). This will be the first pants decomposition.

It is convenient to fix a system of geometric generators $c_1,\ldots,c_n$ of $\pi_1\surface$ compatible with $\mathcal{B}$, i.e.~such the pants curves of $\mathcal{B}$ are represented by the fundamental group elements $b_i=(c_1\cdots c_{i+1})^{-1}$ as on Figure~\ref{fig:standard-pants-decomposition}.

In order to prove Proposition~\ref{prop:transversality}, we have to find a second pants decomposition $\mathcal{D}$ of $\surface$ with induced moment map $\delta\colon\RepDT\to\Delta_{\mathcal{D}}$ such that the cotangent space to $\RepDT$ at $[\rho]$ is generated by $(d\beta_i)_{[\rho]}$ and $(d\delta_i)_{[\rho]}$ for $i=1,\ldots, n-3$. To establish this, it is enough to show that
\begin{itemize}
    \item $\{\beta_i,\delta_i\}([\rho])\neq 0$ for $i=1,\ldots, n-3$, and
    \item $\{\beta_i,\delta_j\}([\rho])=0$ whenever $i>j$,
\end{itemize}
where $\{\cdot,\cdot\}$ denotes the Poisson bracket induced by the Goldman symplectic form on $\RepDT$.

\begin{rmk}[Calculations on sub-spheres]\label{rem:sub-spheres}
    Given $S'\subset S$ a sub-sphere with at least four punctures, the restriction of any DT representation $\rho\colon\pi_1S\to\psl$ to $\pi_1S'\subset \pi_1S$ defines a new DT representation $\rho'$. Let $\RepDT$ and $\RepDTarg{\alpha'}{S'}$ be the DT components containing $[\rho]$ and $[\rho']$. Now, let $a_1$ and $a_2$ be two non-peripheral simple closed curves on $S'$ with associated angle functions $\vartheta_1',\vartheta_2' \colon\RepDTarg{\alpha'}{S'}\to (0,2\pi)$.
    Since $a_1$ and $a_2$ are also simple closed curves on $S$, they also induce angle functions $\vartheta_1,\vartheta_2 \colon\RepDT\to (0,2\pi)$. As explained by Camacho-Cadena--Farre--Wienhard in~\cite[Theorem~B]{subsurfaces}, the Hamiltonian flows of $\vartheta_1$ and $\vartheta_2$ are \emph{sub-surface deformations} along $S'$, meaning that the Hamiltonian flows do not modifiy representations, up to conjugation, on the complement of $S'$ in $S$. In particular, it holds that
    \begin{equation}\label{eq:sub-spheres}
        \{\vartheta_1,\vartheta_2\}([\rho])=\{\vartheta_1',\vartheta_2'\}([\rho']).
    \end{equation}
\end{rmk}

We will construct $\mathcal{D}$ by selecting one curve after another, using our knowledge of the 4-punctured case. When $n=4$, the pants decompositions $\mathcal{B}$ and $\mathcal{D}$ each consists of a single pants curve represented by a fundamental group element $b_1=(c_1c_2)^{-1}$, respectively $d_1$, as depicted on Figure~\ref{fig:pants-decompositions-n=4}. It was proved in~\cite[Lemma~3.3]{orbit-closures} that $d_1$ can be taken to be either $(c_2c_3)^{-1}$ or $(c_1c_3)^{-1}$ (depending on $[\rho]$) to ensure transversality. 
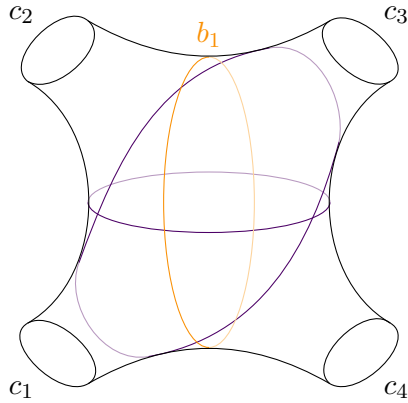
\begin{figure}[h]
\begin{tikzpicture}[scale=.8]  
  \draw[lightplum] (-2,0) arc(180:0:2 and .5);
  \draw[plum] (-2,0) arc(180:0:2 and -.5);

  \draw[lightplum] (1,2.55) to[out=10, in=80] (2.15,1);
  \draw[plum] (1,2.55) to[out=190, in=70] (-2.15,-1);
  \draw[lightplum] (-2.15,-1) to[out=-110, in=190] (-1,-2.55);
  \draw[plum] (-1,-2.55) to[out=10, in=-105] (2.15,1);

  \draw[apricot] (0,2.4) arc(90:270:.75 and 2.4) node[at start, above]{$b_1$};
  \draw[lightapricot] (0,2.4) arc(90:270:-.75 and 2.4);
  
  \draw (-2,3) to[out=-30,in=210] (2,3);
  \draw (3,2) to[out=210,in=150] (3,-2);
  \draw (2,-3) to[out=150,in=30] (-2,-3);
  \draw (-3,-2) to[out=30,in=-30] (-3,2);

  \draw (-3,2) to[out=-30,in=-40] (-2,3);
  \draw (-3,2) to[out=140,in=140] (-2,3);
  \draw (2,3) to[out=30,in=30] (3,2);
  \draw (2,3) to[out=210,in=210] (3,2);
  \draw (3,-2) to[out=-30,in=-30] (2,-3);
  \draw (3,-2) to[out=140,in=140] (2,-3);
  \draw (-2,-3) to[out=30,in=30] (-3,-2);
  \draw (-2,-3) to[out=210,in=210] (-3,-2); 

  \draw (-3.1,3.1) node{$c_2$};
  \draw (3.1,3.1) node{$c_3$};
  \draw (3.1,-3.1) node{$c_4$};
  \draw (-3.1,-3.1) node{$c_1$};
\end{tikzpicture}
\caption{The curve $b_1=(c_1c_2)^{-1}$ on a $4$-punctured sphere in orange and the two candidate curves for $d_1$ in mauve.}
\label{fig:pants-decompositions-n=4}
\end{figure}

\begin{fact}[{\cite[Lemma~3.3]{orbit-closures}}]\label{fact:transversality-n=4}
    Let $\Sigma$ denote a 4-punctured sphere with fundamental group 
    \[
    \pi_1\Sigma = \big\langle \overline{c}_1, \overline{c}_2, \overline{c}_3, \overline{c}_4 : \overline{c}_1\overline{c}_2\overline{c}_3\overline{c}_4=1\big\rangle,
    \]
    and $\overline{\mathcal{B}}$ be the pants decomposition of $\Sigma$ whose only pants curve is represented by $\overline{b}_1=(\overline{c}_1\overline{c}_2)^{-1}$. Let $[\rho]\in\RepDTarg{\overline{\alpha}}{\Sigma}$. If the $\overline{\mathcal{B}}$-triangle chain of $\rho$ is regular, then one of the simple closed curves represented by $(\overline{c}_2\overline{c}_3)^{-1}$ or $(\overline{c}_1\overline{c}_3)^{-1}$ defines a second pants decomposition $\overline{\mathcal{D}}$ of $\Sigma$ such that 
    \[
    \left\{\overline{\beta}_1,\overline{\delta}_1\right\}([\rho])\neq 0,
    \]
    where $\beta_1,\delta_1\colon\RepDTarg{\overline{\alpha}}{\Sigma}\to (0,2\pi)$ are the moment maps induced by $\mathcal{B}$ and $\mathcal{D}$.
\end{fact}

We will now prove Proposition~\ref{prop:transversality} for an arbitrary number $n\geq 4$ of punctures on $\surface$ by applying Fact~\ref{fact:transversality-n=4} several times.

\subsection{Iterated construction}\label{sec:iterated-construction}
We construct the pants curves $d_1,\ldots,d_{n-3}$ of $\mathcal{D}$ one after another.
\subsubsection{The first curve}
We start the sub-sphere $\surface^{(1)}\subset \surface$ with peripheral curves $(c_1,c_2,c_3,b_2)$ illustrated on Figure~\ref{fig:first-sub-sphere}. 

\begin{figure}[h]
    \centering
    \begin{tikzpicture}[scale=1.3, decoration={
    markings,
    mark=at position 0.6 with {\arrow{>}}}]
    \draw[postaction={decorate}] (0,-.5) arc(-90:-270: .25 and .5) node[midway, left]{$c_1$};
    \draw[black!40] (0,.5) arc(90:-90: .25 and .5);
    \draw[apricot] (2,.5) arc(90:270: .25 and .5) node[midway, left]{$b_1$};
    \draw[lightapricot] (2,.5) arc(90:-90: .25 and .5);
   
    \draw[black!40] (6,.5) arc(90:270: .25 and .5) node[midway, left]{$b_3$};
    \draw[black!40] (6,.5) arc(90:-90: .25 and .5);
    \draw[black!40] (8,.5) arc(90:270: .25 and .5); 
    \draw[black!40] (8,.5) arc(90:-90: .25 and .5);
  
    \draw (.5,1) arc(180:0: .5 and .25) node[midway, above]{$c_2$};
    \draw[postaction={decorate}] (.5,1) arc(-180:0: .5 and .25);
    \draw (2.5,1) arc(180:0: .5 and .25)node[midway, above]{$c_3$};
    \draw[postaction={decorate}] (2.5,1) arc(-180:0: .5 and .25);
    \draw[black!40] (4.5,1) arc(180:0: .5 and .25)node[midway, above]{$c_4$};
    \draw[black!40, postaction={decorate}] (4.5,1) arc(-180:0: .5 and .25);
    \draw[black!40] (6.5,1) arc(180:0: .5 and .25)node[midway, above]{$c_5$};
    \draw[black!40, postaction={decorate}] (6.5,1) arc(-180:0: .5 and .25);
    \draw (4,.5) arc(90:270: .25 and .5) node[midway, left]{$b_2$};
    \draw (4,.5) arc(90:-90: .25 and .5);
    \draw (8.7, 0) node{$\ldots$};
   
    \draw (0,.5) to[out=0,in=-90] (.5,1);
    \draw (1.5,1) to[out=-90,in=180] (2,.5);
    \draw (0,-.5) to[out=0,in=180] (2,-.5);
  
    \draw (2,.5) to[out=0,in=-90] (2.5,1);
    \draw (3.5,1) to[out=-90,in=180] (4,.5);
    \draw (2,-.5) to[out=0,in=180] (4,-.5);
  
    \draw[black!40] (4,.5) to[out=0,in=-90] (4.5,1);
    \draw[black!40] (5.5,1) to[out=-90,in=180] (6,.5);
    \draw[black!40] (4,-.5) to[out=0,in=180] (6,-.5);
  
    \draw[black!40] (6,.5) to[out=0,in=-90] (6.5,1);
    \draw[black!40] (7.5,1) to[out=-90,in=180] (8,.5);
    \draw[black!40] (6,-.5) to[out=0,in=180] (8,-.5);
    \end{tikzpicture}
    \caption{The sub-sphere $\surface^{(1)}$.}
    \label{fig:first-sub-sphere}
    \end{figure}
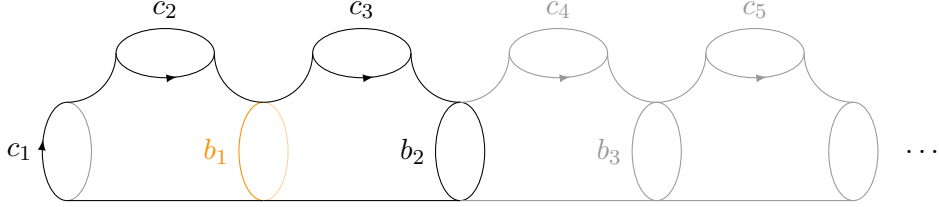

The simple closed curve represented by $b_1$ defines a pants decomposition $\mathcal{B}_1$ of $\surface^{(1)}$. The restriction $\rho^{(1)}$ of $\rho$ to $\pi_1\surface^{(1)}$ is still a DT representation with regular $\mathcal{B}_1$-triangle chain (it is a sub-chain of the $\mathcal{B}$-triangle chain of $\rho$ which is regular). Its conjugacy class belongs to the DT component $\RepDTarg{\alpha^{(1)}}{\surface^{(1)}}$, where $\alpha^{(1)}$ is the angle vector $(\alpha_1,\alpha_2,\alpha_3,\beta_2([\rho]))$. Applying Fact~\ref{fact:transversality-n=4} to $\rho^{(1)}$, we obtain a first pants curve $d_1$ on $\surface^{(1)}$ (which is either equal to $(c_2c_3)^{-1}$ or $(c_1c_3)^{-1}$, see Figure~\ref{fig:second-sub-sphere}) such that $\{\beta_1,\delta_1\}([\rho^{(1)}])\neq 0$. By Equation~\eqref{eq:sub-spheres} from Remark~\ref{rem:sub-spheres}, we conclude that $\{\beta_1,\delta_1\}([\rho])=\{\beta_1,\delta_1\}([\rho^{(1)}])\neq 0$. Furthermore, $\{\beta_i,\delta_1\}([\rho])=0$ for $i>1$ because the curve $d_1$ is disjoint from the curves $b_i$ for $i>1$.

\subsubsection{The second curve}
We now consider the next sub-sphere $\surface^{(2)}\subset S$ illustrated on Figure~\ref{fig:second-sub-sphere} with peripheral curves being either
\begin{itemize}
    \item $(c_1,d_1^{-1},c_4,b_3)$ if $d_1=(c_2c_3)^{-1}$, or
    \item $(c_1c_2c_1^{-1}, d_1^{-1}, c_4, b_3)$ if $d_1=(c_1c_3)^{-1}$.
\end{itemize}
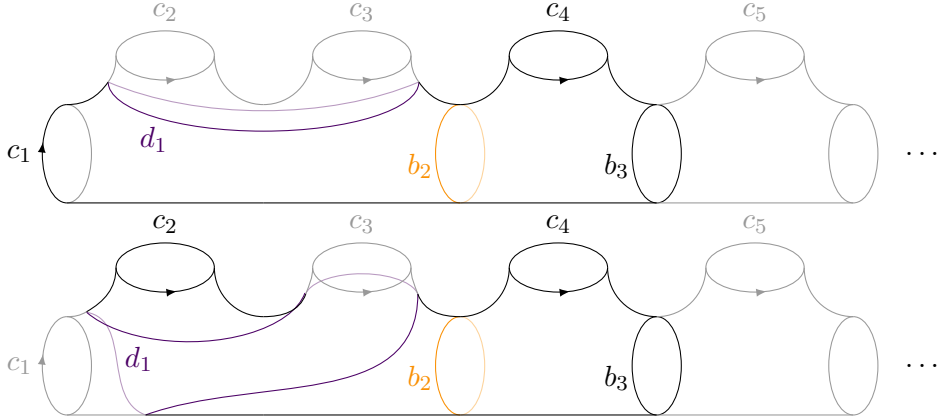
\begin{figure}[h]
    \centering
    \begin{tikzpicture}[scale=1.3, decoration={
    markings,
    mark=at position 0.6 with {\arrow{>}}}]
  \draw[postaction={decorate}] (0,-.5) arc(-90:-270: .25 and .5) node[midway, left]{$c_1$};
   \draw[black!40] (0,.5) arc(90:-90: .25 and .5);
   \draw[apricot] (4,.5) arc(90:270: .25 and .5) node[near end, above left]{$b_2$};
   \draw[lightapricot] (4,.5) arc(90:-90: .25 and .5);
   \draw (6,.5) arc(90:270: .25 and .5) node[near end, above left]{$b_3$};
   \draw (6,.5) arc(90:-90: .25 and .5);
  \draw[black!40] (8,.5) arc(90:270: .25 and .5);
   \draw[black!40] (8,.5) arc(90:-90: .25 and .5);
   \draw (8.7, 0) node{$\ldots$};
  
   \draw[black!40] (.5,1) arc(180:0: .5 and .25) node[midway, above]{$c_2$};
  \draw[black!40, postaction={decorate}] (.5,1) arc(-180:0: .5 and .25);
   \draw[black!40] (2.5,1) arc(180:0: .5 and .25)node[midway, above]{$c_3$};
   \draw[black!40, postaction={decorate}] (2.5,1) arc(-180:0: .5 and .25);
   \draw (4.5,1) arc(180:0: .5 and .25)node[midway, above]{$c_4$};
   \draw[postaction={decorate}] (4.5,1) arc(-180:0: .5 and .25);
   \draw[black!40] (6.5,1) arc(180:0: .5 and .25)node[midway, above]{$c_5$};
   \draw[black!40, postaction={decorate}] (6.5,1) arc(-180:0: .5 and .25);

   \draw[plum] (3.58,.73) arc(0:180:1.58 and -.5) node[near end, below]{$d_1$};
   \draw[lightplum] (3.58,.73) arc(-50:-130:2.46 and 1.25);
   
   \draw (0,.5) to[out=0,in=-125] (.42,.73);
   \draw[black!40] (.42,.73) to[out=55,in=-90] (.5,1);
   \draw[black!40] (1.5,1) to[out=-90,in=180] (2,.5);
   \draw (0,-.5) to[out=0,in=180] (2,-.5);
  
   \draw[black!40] (2,.5) to[out=0,in=-90] (2.5,1);
   \draw[black!40] (3.5,1) to[out=-90,in=125] (3.58,.73);
   \draw (3.58,.73) to[out=-55,in=180] (4,.5);
   \draw (2,-.5) to[out=0,in=180] (4,-.5);
  
   \draw (4,.5) to[out=0,in=-90] (4.5,1);
   \draw (5.5,1) to[out=-90,in=180] (6,.5);
   \draw (4,-.5) to[out=0,in=180] (6,-.5);
  
   \draw[black!40] (6,.5) to[out=0,in=-90] (6.5,1);
   \draw[black!40] (7.5,1) to[out=-90,in=180] (8,.5);
   \draw[black!40] (6,-.5) to[out=0,in=180] (8,-.5);
 \end{tikzpicture}
 \bigskip
 \begin{tikzpicture}[scale=1.3, decoration={
     markings,
     mark=at position 0.6 with {\arrow{>}}}]
   \draw[black!40, postaction={decorate}] (0,-.5) arc(-90:-270: .25 and .5) node[midway, left]{$c_1$};
   \draw[black!40] (0,.5) arc(90:-90: .25 and .5);
   \draw[apricot] (4,.5) arc(90:270: .25 and .5) node[near end, above left]{$b_2$};
   \draw[lightapricot] (4,.5) arc(90:-90: .25 and .5);
   \draw (6,.5) arc(90:270: .25 and .5) node[near end, above left]{$b_3$};
   \draw (6,.5) arc(90:-90: .25 and .5);
   \draw[black!40] (8,.5) arc(90:270: .25 and .5);
   \draw[black!40] (8,.5) arc(90:-90: .25 and .5);
   \draw (8.7, 0) node{$\ldots$};
  
   \draw[lightplum] (3.57,.73) arc(170:10: -.58 and .25);
   \draw[lightplum] (.2,.545) edge[out=-10,in=160,-] (.8,-.5);
   \draw[plum] (.2,.545) arc(-150:-10: 1.2 and .6) node[near start, below]{$d_1$} (.8,-.5);
   \draw[plum] (3.57,.73) edge[out=270,in=20,-] (.8,-.5);
   \draw (.5,1) arc(180:0: .5 and .25) node[midway, above]{$c_2$};
   \draw[postaction={decorate}] (.5,1) arc(-180:0: .5 and .25);
   \draw[black!40] (2.5,1) arc(180:0: .5 and .25)node[midway, above]{$c_3$};
   \draw[black!40, postaction={decorate}] (2.5,1) arc(-180:0: .5 and .25);
   \draw (4.5,1) arc(180:0: .5 and .25)node[midway, above]{$c_4$};
   \draw[postaction={decorate}] (4.5,1) arc(-180:0: .5 and .25);
   \draw[black!40] (6.5,1) arc(180:0: .5 and .25)node[midway, above]{$c_5$};
   \draw[black!40, postaction={decorate}] (6.5,1) arc(-180:0: .5 and .25);
   
   \draw[black!40] (0,.5) to[out=0,in=-150] (.2,.55);
   \draw (.2,.55) to[out=30,in=-90] (.5,1);
   \draw (1.5,1) to[out=-90,in=180] (2,.5);
   \draw[black!40] (0,-.5) to[out=0,in=180] (.8,-.5);
   \draw (.8,-.5) to[out=0,in=180] (2,-.5);
  
   \draw (2,.5) to[out=0,in=-110] (2.43,.73);
   \draw[black!40] (2.43,.73) to[out=70,in=-90] (2.5,1);
   \draw[black!40] (3.5,1) to[out=-90,in=110] (3.57,.73);
   \draw (3.57,.73) to[out=-70,in=180] (4,.5);
   \draw (2,-.5) to[out=0,in=180] (4,-.5);
  
   \draw (4,.5) to[out=0,in=-90] (4.5,1);
   \draw (5.5,1) to[out=-90,in=180] (6,.5);
   \draw (4,-.5) to[out=0,in=180] (6,-.5);
  
   \draw[black!40] (6,.5) to[out=0,in=-90] (6.5,1);
   \draw[black!40] (7.5,1) to[out=-90,in=180] (8,.5);
   \draw[black!40] (6,-.5) to[out=0,in=180] (8,-.5);
 \end{tikzpicture}
    \caption{Above: the sub-sphere $\surface^{(2)}$ when $d_1=(c_2c_3)^{-1}$. Below: the sub-sphere $\surface^{(2)}$ when $d_1=(c_1c_3)^{-1}$.}
    \label{fig:second-sub-sphere}
\end{figure}

Note that the product of the four peripheral curves is always trivial, and so they give a system of geometric generators of $\pi_1\surface^{(2)}$. Also note that $b_2^{-1}$ is always the product of the first two peripheral curves, so that $b_2$ defines a pants decomposition $\mathcal{B}_2$ of $\surface^{(2)}$. In order to apply Fact~\ref{fact:transversality-n=4} again, we have to check that the restriction $\rho^{(2)}$ of $\rho$ to $\pi_1\surface^{(2)}$ has a regular $\mathcal{B}_2$-triangle chain, which we denote by $\mathcal{T}_2$. Conveniently, the second triangle in $\mathcal{T}_2$---the one whose vertices are the fixed points of $\rho(b_2), \rho(c_4),\rho(b_3)$---is also a triangle in the $\mathcal{B}$-triangle chain of $[\rho]$, hence is non-degenerate. Two of the vertices of the first triangle in $\mathcal{T}_2$ are the fixed points of $\rho(d_1)$ and $\rho(b_2)$, which we denote by $D_1$ and $B_2$. So, $\mathcal{T}_2$ is regular if and only if $D_1\neq B_2$. We prove that this is always the case.
\begin{claim}\label{claim:1-Di-neq-Bi+1}
    If $d_1=(c_2c_3)^{-1}$, then $D_1\neq B_2$.
\end{claim}
\begin{proofclaim}
    The situation is similar to the one of Claim~\ref{claim:exclude-0-pi} (with $i=2$ and $j=3$, and $Y_1=D_1$). If $D_1=B_2$, then the triangles $(B_1,C_3,B_2)$ and $(C_2,C_3,D_1)$ both have an angle $\pi-\alpha_3/2$ at $C_3$, which forces $C_3,B_2,C_2$ to be collinear. This implies $\gamma_1([\rho])=\pi$ or $\gamma_1([\rho])=0$ (actually $\gamma_1([\rho])=0$ is impossible, but we do not this observation here). The configuration with $\gamma_1([\rho])=\pi$ is illustrated on Figure~\ref{fig:triangle-chain-gamma=pi-bis}. This would however imply that $\{\beta_1,\delta_1\}([\rho])=0$ by~\cite[Lemma~3.10]{orbit-closures}, a contradiction.
\end{proofclaim}

    \begin{figure}[h]
    \centering
    \begin{tikzpicture}[font=\sffamily, scale=1.1]
    \node[anchor=south west,inner sep=0] at (0,0) {\includegraphics[width=10.8cm]{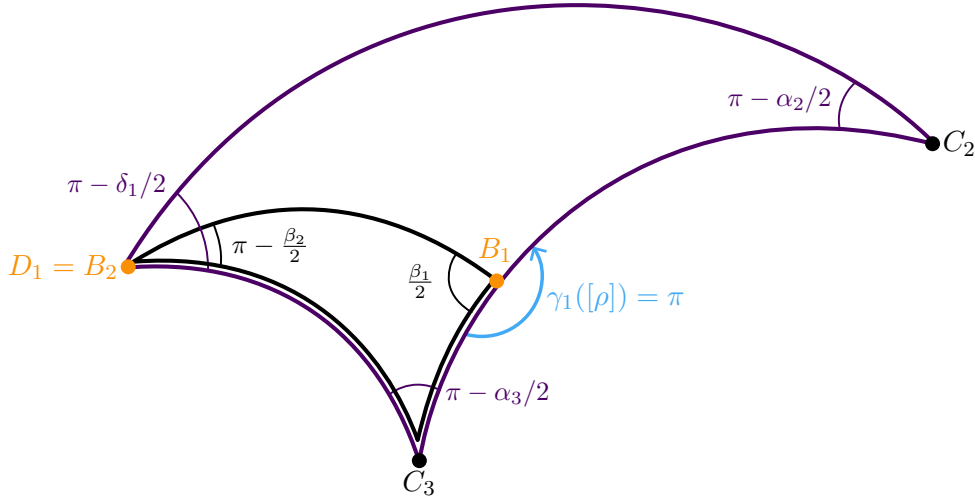}};

    \begin{scope}
    \fill (9.76,3.85) circle (0.09) node[right]{$C_2$};
    \fill (3.59,0.04) circle (0.09) node[below]{$C_3$};
    \draw (1.8, 2.63) node{\small $\pi-\frac{\beta_{2}}{2}$};
    \draw (3.6, 2.2) node{\small $\frac{\beta_{1}}{2}$};
    \end{scope}

    \begin{scope}[plum]
    \draw (4.53, .85) node{\small $\pi-\alpha_{3}/2$};
    \draw (-.05, 3.35) node{\small $\pi-\delta_{1}/2$}; 
    \draw (7.92, 4.38) node{\small $\pi-\alpha_{2}/2$};
    \end{scope}

    \begin{scope}[apricot]
    \fill (0.09,2.37) circle (0.09) node[left]{$D_1=B_{2}$};
    \fill (4.52,2.2) circle (0.09);
    \draw[anchor=south] (4.5, 2.3) node{$B_1$};
    \end{scope}

    \begin{scope}[sky]
    \draw (5,2) node[right]{$\gamma_1([\rho])=\pi$};
    \end{scope}
    \end{tikzpicture}
    \caption{The configuration in which $D_1 = B_2$, forcing $C_2,B_1,C_3$ to be collinear. The triangles are drawn slightly offset to highlight their superposition.}
    \label{fig:triangle-chain-gamma=pi-bis}
    \end{figure}

\begin{claim}\label{claim:2-Di-neq-Bi+1}
    If $d_1=(c_1c_3)^{-1}$, then $D_1\neq B_2$.
\end{claim}
\begin{proofclaim}
    In that case, the first triangle in $\mathcal{T}_2$ has vertices $(\rho(c_1)C_2, D_1, B_2)$, where $C_2$ is the fixed point of $\rho(c_2)$. So, if $D_1=B_2$, then that triangle would be degenerate to a single point and thus $B_2=\rho(c_1)C_2$. This would force the first two triangles in the $\mathcal{B}$-triangle chain of $[\rho]$ to be in configuration illustrated on Figure~\ref{fig:triangle-chain-gamma=beta2}, implying $\gamma_1([\rho])=\beta_1([\rho])/2$, and also $\{\beta_1,\delta_1\}([\rho])=0$ by~\cite[Lemma~3.10]{orbit-closures}. This is again a contradiction. 
    \end{proofclaim}
    
    \begin{figure}[h]
    \centering
    \begin{tikzpicture}[font=\sffamily, scale=1.1]
    \node[anchor=south west,inner sep=0] at (0,0) {\includegraphics[width=10.8cm]{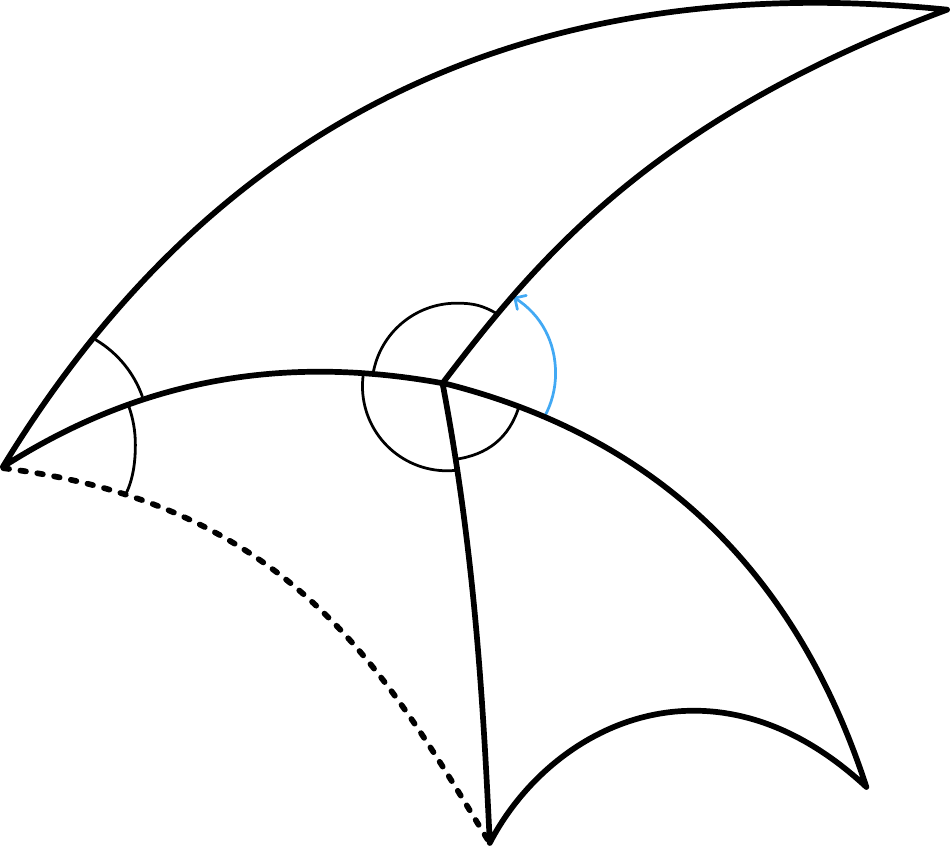}};

    \begin{scope}
    \fill (0.04,3.93) circle (0.07) node[left]{$C_1$};
    \fill (9.8,8.65) circle (0.07) node[right]{$C_2$};
    \fill (8.95,.65) circle (0.07) node[right]{$C_3$};
    \draw (4.3, 5.85) node{\small $\pi-\beta_1/2$};
    \draw (3.6, 3.8) node{\small $\pi-\beta_1/2$};
    \draw (5.35, 3.95) node{\small $\beta_1/2$};
    \draw (1.9, 5.2) node{\small $\pi-\alpha_1/2$};
    \draw (2.1, 4.2) node{\small $\pi-\alpha_1/2$};
    \end{scope}

    \begin{scope}[apricot]
    \fill (5.08, 0.05) circle (0.09) node[below]{$B_2=\rho(c_1)C_2$};
    \fill (4.59,4.76) circle (0.09) node[below left]{$B_1$};
    \end{scope}

    \begin{scope}[sky]
    \draw (5.73,5.1) node[right]{$\gamma_1([\rho])=\beta_1/2$};
    \end{scope}
    \end{tikzpicture}
    \caption{The configuration in which $B_2 = \rho(c_1)C_2$, forcing $\gamma_1([\rho])=\beta_1/2$.}
    \label{fig:triangle-chain-gamma=beta2}
    \end{figure}
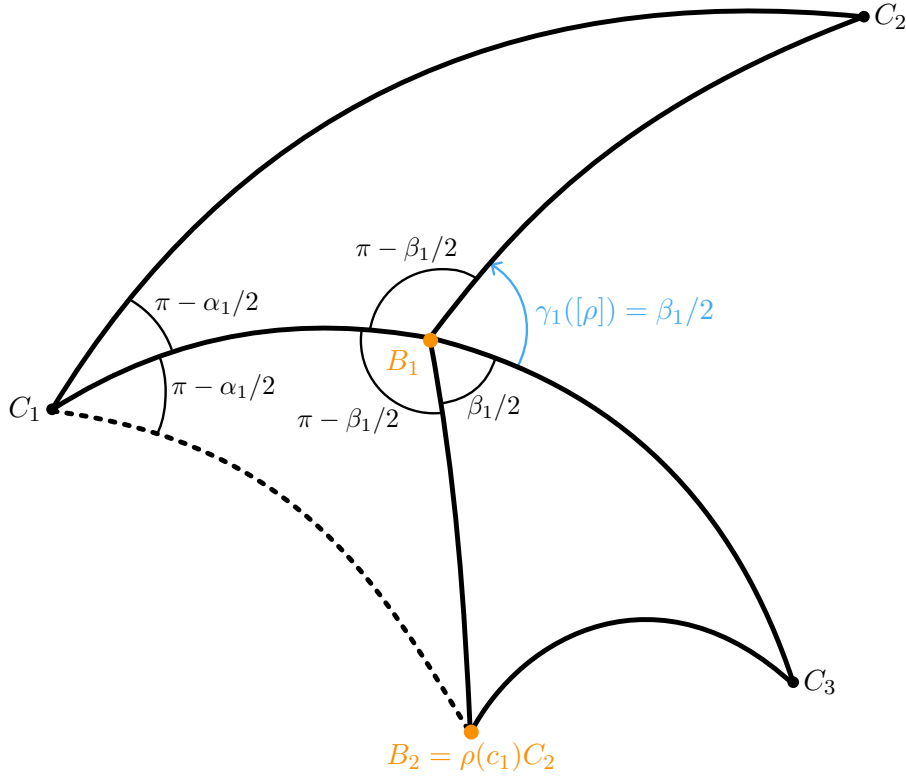  

We conclude that $\mathcal{T}_2$ is regular. We may now apply Fact~\ref{fact:transversality-n=4} to $\rho^{(2)}$ to get the next pants curve $d_2$, which is one of the following four curves:
\begin{align*}
    \text{If $d_1=(c_2c_3)^{-1}$, then } &d_2=\begin{cases}
        (d_1^{-1}c_4)^{-1}=(c_2c_3c_4)^{-1}, & \text{or}\\
        (c_1c_4)^{-1}.
    \end{cases}\\
    \text{If $d_1=(c_1c_3)^{-1}$, then } &d_2= \begin{cases}
        (d_1^{-1}c_4)^{-1}=(c_1c_3c_4)^{-1}, & \text{or}\\
        (c_1c_2c_1^{-1}c_4)^{-1}.
    \end{cases}
\end{align*}
In all four cases, we have $\{\beta_2,\delta_2\}([\rho^{(2)}])\neq 0$. As before, by~\eqref{eq:sub-spheres}, $\{\beta_2,\delta_2\}([\rho])=\{\beta_2,\delta_2\}([\rho^{(2)}])\neq 0$ and $\{\beta_i,\delta_2\}([\rho])=0$ for $i>2$. Note that since $d_2$ is curve on $\surface^{(2)}$ and $d_1$ is a peripheral curve of $\surface^{(2)}$, they are disjoint.

\subsubsection{The other curves}
Assume that we have just constructed the simple closed curve $d_{i-1}$ for $i\geq 3$ on a sub-sphere $\surface^{(i-1)}$ of $\surface$ with the property that $\{\beta_{i-1},\delta_{i-1}\}([\rho])\neq 0$. The next sub-sphere $\surface^{(i)}\subset\surface$ we consider has peripheral curves $(x,d_{i-1}^{-1},c_{i+2},b_{i+1})$, where $x$ is either a conjugate of a peripheral curve of $\surface$, or a conjugate of $d_{i-2}^{-1}$. The product of the first two peripheral curves is $b_i^{-1}$ which give a pants decomposition $\mathcal{B}_i$ of $\surface^{(i)}$. As before, the $\mathcal{B}_i$-triangle chain of the restriction $\rho^{(i)}$ of $\rho$ to $\pi_1\surface^{(i)}$ is regular if and only if $D_{i-1}\neq B_i$, where $D_{i-1}$ and $B_i$ are the fixed points of $\rho(d_{i-1})$ and $\rho(b_i)$. The analogues of Claims~\ref{claim:1-Di-neq-Bi+1} and~\ref{claim:2-Di-neq-Bi+1} hold and show that indeed $D_{i-1}\neq B_i$, because $\{\beta_{i-1},\delta_{i-1}\}([\rho])\neq 0$ by assumption.

We may thus apply Fact~\ref{fact:transversality-n=4} to $\rho^{(i)}$ to find the next simple closed curve $d_i$ which is either $(d_{i-1}^{-1}c_{i+2})^{-1}$ or $(xc_{i+2})^{-1}$. In both cases, the curve $d_i$ is disjoint from all the previously constructed curves $d_j$ for $j<i$ since it lives on $\surface^{(i)}$. Furthermore, $\{\beta_i,\delta_i\}([\rho])\neq 0$ by Fact~\ref{fact:transversality-n=4} and~\eqref{eq:sub-spheres}, and $\{\beta_j,\delta_i\}([\rho])= 0$ for $j>i$, for the curves $b_j$ and $d_i$ are disjoint.

\subsubsection{Conclusion}
The previous steps produce a pants decomposition $\mathcal{D}$ of $\surface$ with induced moment map $\delta\colon\RepDT\to\Delta_\mathcal{D}$ satisfying $\{\beta_i,\delta_i\}([\rho])\neq 0$ for $i=1,\ldots,n-3$ and $\{\beta_i,\delta_j\}([\rho])=0$ for $i>j$. Also note that the pants decomposition $\mathcal{D}$ is chained, as required, because the first $n-3$ pair of pants it determines have peripheral curves $(\ast, c_{i+2}, d_i)$ (hence all contain a puncture of $\surface$) and $c_n$ is a peripheral curve of the last pair of pants. This concludes the proof of Proposition~\ref{prop:transversality}.
\section{Proof of Theorem~\ref{main-thm}}\label{sec:proof}
\subsection{Overview}
We now turn to the proof of Theorem~\ref{main-thm}, which classifies ergodic probability measures on DT components. The argument proceeds as follows: we begin with some elementary observations (Section~\ref{sec:elementary-observations}), then use unique ergodicity of irrational torus rotations together with disintegration to show that most conditional measures are Lebesgue measures (Sections~\ref{sec:dynamics-on-fibers} and~\ref{sec:disintegration}), and finally conclude via a transversality argument (Section~\ref{sec:transversality-moment-maps}).

\subsection{First observations}\label{sec:elementary-observations}
In order to prove Theorem~\ref{main-thm}, we fix a DT component $\RepDT$. Its Goldman measure will be denoted by $\nuG$. Let $\mu$ be any Borel probability measure on $\RepDT$, which is ergodic for the action of the pure mapping class group $\PMod(\surface)$.

First, note that if $\mu$ has atoms, then $\mu$ is the counting measure along a finite orbit. From now on, let us assume that $\mu$ has no atoms. We will prove that $\mu=\nuG$. We start with two elementary claims.
\begin{claim}\label{claim:gzz-mu<<nuG}
Since $\nuG$ is ergodic and $\mu$ is invariant, it suffices to show that $\mu \ll \nuG$ (i.e.~$\mu$ is absolutely continuous with respect to $\nuG$) to conclude that $\mu=\nuG$.
\end{claim}
\begin{proofclaim}
    If $\mu\ll \nuG$, then by the Radon--Nikodym theorem, there exists a measurable, non-negative function $f=d\mu/d\nuG$ such that
    \[
    \mu=\int f\, d\nuG.
    \]
    Invariance of $\mu$ and $\nuG$ implies invariance $\nuG$-almost everywhere of $f$. Since $\mu$ is ergodic, $f$ is thus constant $\nuG$-almost everywhere, and this constant must be $1$ because $\mu$ and $\nuG$ are probability measures. Hence, $\mu=\nuG$.
\end{proofclaim}

The second claim will allow us to work locally around every point. If $(X,m)$ is a Borel probability space and $U\subset X$ is a nonempty open set with $m(U)\neq 0$, then we call the Borel probability measure $m_U$ defined by $m_U(A)=m(A\cap U)/m(U)$ the \emph{localization} of $m$ on $U$. 

\begin{claim}\label{claim:gzz-muU<<nuGU}
To prove that $\mu \ll \nuG$, it suffices to show that every point of $\RepDT$ has an open neighborhood $U$ such that $\mu(U)=0$ or $\mu_U \ll (\nuG)_U$.
\end{claim}
\begin{proofclaim}
By compactness of DT components, there is a finite covering of $\RepDT$ given by finitely many such open sets $U_1,\ldots, U_m$. Now, if $\nuG(A)=0$, then $\nuG(A\cap U_i)=0$ and thus $\mu(A\cap U_i)=0$ in both cases. Hence, $\mu(A)\leq \sum_i\mu(A\cap U_i)=0$, which shows $\mu\ll \nuG$.
\end{proofclaim}

\begin{rmk}\label{rem:independence-from-minimality}
It was shown in~\cite{orbit-closures} that the only closed invariant subsets of $\RepDT$ are either finite orbits or the whole space $\RepDT$. In particular, since $\mu$ is assumed to be atom-free, it must have full support, and hence $\mu(U)>0$ for every nonempty open subset $U$. This observation would allow us to simplify the statement of Claim~\ref{claim:gzz-muU<<nuGU}. However, in order to prove Theorem~\ref{main-thm} independently of~\cite{orbit-closures}, we choose not to rely on this result.
\end{rmk}

\subsection{Dynamics on the fibers}\label{sec:dynamics-on-fibers}
Let $\mathcal{B}$ be a pants decomposition of $\surface$, with induced moment map $\beta\colon \RepDT \to \Delta_{\mathcal{B}}$. Its fibers are isotropic tori: those lying over the interior $\oset{\circ}{\Delta}_{\mathcal{B}}$ of $\Delta_\mathcal{B}$ are Lagrangian, while the others have lower dimension. All fibers are invariant under the action of the free abelian subgroup $\Lambda_{\mathcal{B}} \subset \PMod(\surface)$ generated by the $n-3$ Dehn twists along the pants curves of $\mathcal{B}$. Over $\oset{\circ}{\Delta}_{\mathcal{B}}$, these Lagrangian tori form a trivial torus bundle, yielding action--angle coordinates on $\beta^{-1}(\oset{\circ}{\Delta}_{\mathcal{B}})$ and a diffeomorphism with $\oset{\circ}{\Delta}_{\mathcal{B}} \times (\mathbb{R}/2\pi\mathbb{Z})^{n-3}$.

Given $\beta_0=(\beta_1,\ldots,\beta_{n-3})\in\,\oset{\circ}{\Delta}_{\mathcal{B}}$, the action of $\Lambda_{\mathcal{B}}$ on $\beta^{-1}(\beta_0)$ is conjugate to the additive action of $\oplus_i \beta_i\Z\cong \Z^{n-3}$ on $(\R/2\pi\Z)^{n-3}$. The induced action of $\Z$ on $(\R/2\pi\Z)^{n-3}$ given by 
\begin{equation}\label{eq:torus-rotation}
R_{\beta_0}(x_1,\ldots,x_{n-3})= (x_1+\beta_1,\ldots,x_{n-3}+\beta_{n-3}).
\end{equation}
corresponds to the action of the 1-parameter subgroup of $\Lambda_\mathcal{B}$ generated by the product of all $n-3$ Dehn twists along the pants curves of $\mathcal{B}$. When $\beta_1,\ldots,\beta_{n-3}$ and $\pi$ are $\Q$-linearly independent, then there is a unique invariant measure on $\beta^{-1}(\beta_0)$ for the action of $R_{\beta_0}$ (see e.g.~\cite[Corollary~4.15]{einsiedler-ward}). It coincides with the conditional measure of $\nuG$ along the fiber $\beta^{-1}(\beta_0)$, which corresponds, via angle coordinates, to the normalized Lebesgue measure $\lambda$ on $(\R/2\pi\Z)^{n-3}$. 

The subset of $\oset{\circ}{\Delta}_{\mathcal{B}}$ consisting of all tuples $(\beta_1,\ldots,\beta_{n-3})$ that are $\Q$-linearly independent with $\pi$ will be denoted by
\[
\Delta^{\mathrm{irr}}_{\mathcal{B}}\subset\, \oset{\circ}{\Delta}_{\mathcal{B}}.
\]
Its complement $\oset{\circ}{\Delta}_{\mathcal{B}}\setminus\Delta^{\mathrm{irr}}_{\mathcal{B}}$ is countable. Since we are assuming that $\mu$ is atom-free, Proposition~\ref{prop:pushforward-no-atoms} implies that $\beta_\ast\mu$ is also atom-free. Therefore, for any measurable set $A\subset \,\oset{\circ}{\Delta}_{\mathcal{B}}$, we have
\begin{equation}\label{eq:push-forward-full-measure-on-Delta^irr}
    \beta_\ast\mu(A\cap \Delta^{\mathrm{irr}}_{\mathcal{B}})=\beta_\ast\mu(A).
\end{equation}

\subsection{Disintegration}\label{sec:disintegration-moment-map}
We can disintegrate $\mu$ along the moment map $\beta$ induced by a pants decomposition $\mathcal{B}$. This produces a system of Borel probability measures $\{\mu_{\beta_0}\} _{\beta_0\in\Delta_\mathcal{B}}$ on $\RepDT$ such that
\begin{equation}\label{eq:disintegration}
    \mu = \int_{\Delta_\mathcal{B}} \mu_{\beta_0} \, d(\beta_\ast\mu)(\beta_0).
\end{equation}
By the properties of disintegration stated in  Section~\ref{sec:disintegration}, there is a $(\beta_\ast\mu)$-full measure subset 
\[
\Delta_{\mathcal{B}}'\subset \Delta_{\mathcal{B}}
\]
such that for every $\beta_0\in\Delta_{\mathcal{B}}'$, the conditional measure $\mu_{\beta_0}$ is supported on the fiber $\beta^{-1}(\beta_0)$ and is $\Lambda_\mathcal{B}$-invariant. 

In particular, for every $\beta_0 \in \Delta_{\mathcal{B}}' \cap \Delta^{\mathrm{irr}}_{\mathcal{B}}$, the unique ergodicity of the rotation $R_{\beta_0}$ from~\eqref{eq:torus-rotation} implies that, for every measurable set $A \subset \beta^{-1}(\beta_0)$,
\begin{equation}\label{eq:conditional-measure-equals-Lebesgue}
    \mu_{\beta_0}(A)=\lambda(A),
\end{equation}
where we identify $\beta^{-1}(\beta_0)$ with $(\mathbb{R}/2\pi\mathbb{Z})^{n-3}$ via angle coordinates. Moreover, since $\Delta_{\mathcal{B}}' \cap \Delta^{\mathrm{irr}}_{\mathcal{B}}$ has full $(\beta_\ast\mu)$-measure in $\Delta_{\mathcal{B}}' \cap \oset{\circ}{\Delta}_{\mathcal{B}}$ by~\eqref{eq:push-forward-full-measure-on-Delta^irr}, it follows that~\eqref{eq:conditional-measure-equals-Lebesgue} holds for $(\beta_\ast\mu)$-almost every $\beta_0 \in \Delta_{\mathcal{B}}' \cap \oset{\circ}{\Delta}_{\mathcal{B}}$.

\subsection{Transversality}\label{sec:transversality-moment-maps}
We will now conduct local investigations, using a transversality argument, to verify the hypothesis of Claim~\ref{claim:gzz-muU<<nuGU}. Namely, we will prove the following.
\begin{lem}\label{lem:local-aboslute-continuity}
For every $[\rho]$ in $\RepDT$, there is an open neighborhood $U$ of $[\rho]$ such that either $\mu(U)=0$ or $\mu_U\ll (\nuG)_U$.
\end{lem}
Once Lemma~\ref{lem:local-aboslute-continuity} has been proved, we may apply Claim~\ref{claim:gzz-muU<<nuGU} and then Claim~\ref{claim:gzz-mu<<nuG} to conclude the proof of Theorem~\ref{main-thm}.

\begin{proof}[Proof of Lemma~\ref{lem:local-aboslute-continuity}]
Let us fix some $[\rho]\in\RepDT$. We first apply Proposition~\ref{prop:transversality} to find two chained pants decompositions $\mathcal{B}$ and $\mathcal{D}$ of $\surface$ such that 
\[
T_{[\rho]}^\ast\RepDT=\big\langle (d\beta_i)_{[\rho]}, (d\delta_i)_{[\rho]} : i=1,\ldots,n-3\big\rangle.
\]
In particular, this means that $\beta([\rho])\in \oset{\circ}{\Delta}_\mathcal{B}$ and $\delta([\rho])\in \oset{\circ}{\Delta}_\mathcal{D}$. Furthermore, by the Inverse Function Theorem, there is an open neighborhood $U$ of $[\rho]$ and an open subset $V_\mathcal{B}\times V_\mathcal{D}$ of $\oset{\circ}{\Delta}_\mathcal{B}\times \oset{\circ}{\Delta}_\mathcal{D}$ such that
\[
(\beta,\delta)\colon U\to V_\mathcal{B}\times V_\mathcal{D}
\]
is a diffeomorphism. We denote by $\lambda_\mathcal{B}$ and $\lambda_\mathcal{D}$ the normalized Lebesgue measures on $V_{\mathcal{B}}$ and $V_\mathcal{D}$. The localized Goldman measure $(\nuG)_U$ is given by a smooth positive density on $U$, therefore $(\beta,\delta)_\ast (\nuG)_U$ is equivalent to the product measure $\lambda_\mathcal{B}\times\lambda_\mathcal{D}$, i.e.~they have the same null sets:
\begin{equation}\label{eq:Goldman-equivalent-Lebesgue}
(\beta,\delta)_\ast (\nuG)_U\sim \lambda_\mathcal{B}\times\lambda_\mathcal{D}.
\end{equation}

Now, assuming $\mu(U)\neq 0$, the localized probability measure $\mu_U$ is a Borel probability measure with support $U$ which can be disintegrated along the moment map $\beta$. Our first observation is that localization and disintegration are two commuting operations.
\begin{claim}\label{claim:localization-commutes-disintegration}
    The conditional measures $\{(\mu_U)_{\beta_0}\}$ coming from the disintegration of $\mu_U$ along $\restr{\beta}{U}\colon U\to V_\mathcal{B}$ coincide with the localized conditional measures $\{(\mu_{\beta_0})_U\}$ for $(\beta_\ast\mu_U)$-almost every $\beta_0\in V_\mathcal{B}$.
\end{claim}
\begin{proofclaim}
    This is a consequence of the uniqueness of conditional measures guaranteed by the properties of disintegration, see Section~\ref{sec:disintegration}.
\end{proofclaim}

We now have all the ingredients for the upcoming computations.

\begin{claim}\label{claim:pushforward-localization-absolutely-continuous-wrt-Lebesgue}
    If $\mu(U)\neq 0$, then $(\beta,\delta)_\ast \mu_U \ll \lambda_{\mathcal{B}}\times\lambda_\mathcal{D}$.
\end{claim}
\begin{proofclaim}
    Let $A_\mathcal{B}\subset V_\mathcal{B}$ and $A_\mathcal{D}\subset V_\mathcal{D}$ be two Borel measurable subsets. We will first assume that $\lambda_\mathcal{D}(A_\mathcal{D})=0$. Disintegration of $\mu_U$ along $\beta$ gives
    \begin{align*}
    (\beta,\delta)_\ast\mu_U(A_\mathcal{B}\times A_\mathcal{D})&=\int_{V_\mathcal{B}} (\mu_U)_{\beta_0}\left(\beta^{-1}(A_\mathcal{B})\cap \delta^{-1}(A_\mathcal{D})\right) \, d(\beta_\ast\mu_U)(\beta_0)\\
    &=\int_{V_\mathcal{B}} (\mu_{\beta_0})_U\left(\beta^{-1}(A_\mathcal{B})\cap \delta^{-1}(A_\mathcal{D})\right) \, d(\beta_\ast\mu_U)(\beta_0)\\
    &=\int_{V_\mathcal{B}} \frac{\mu_{\beta_0}\left(\beta^{-1}(A_\mathcal{B})\cap \delta^{-1}(A_\mathcal{D})\cap U\right)}{\mu_{\beta_0}\left(U\right)} \, d(\beta_\ast\mu_U)(\beta_0),\\
    \end{align*}
    where we used Claim~\ref{claim:localization-commutes-disintegration} in the second equality, and the definition of localized measure in the third equality. Recall that $\Delta_{\mathcal{B}}^{\mathrm{irr}} \cap \Delta_{\mathcal{B}}' \cap V_{\mathcal{B}}$ has full $(\beta_\ast\mu)$-measure in $V_{\mathcal{B}}$, and hence $\beta_\ast\mu_U\big(\Delta_{\mathcal{B}}^{\mathrm{irr}} \cap \Delta_{\mathcal{B}}' \cap V_{\mathcal{B}}\big)=1$. We may therefore write
    \begin{align*}
    (\beta,\delta)_\ast\mu_U(A_\mathcal{B}\times A_\mathcal{D})&=\int_{\Delta_\mathcal{B}'\cap V_\mathcal{B}} \frac{\mu_{\beta_0}\left(\beta^{-1}(A_\mathcal{B})\cap \delta^{-1}(A_\mathcal{D})\cap U\right)}{\mu_{\beta_0}\left(U\right)} \, d(\beta_\ast\mu_U)(\beta_0)\\
    &=\int_{\Delta_\mathcal{B}'\cap A_\mathcal{B}} \frac{\mu_{\beta_0}\left(\beta^{-1}(\beta_0)\cap \delta^{-1}(A_\mathcal{D})\cap U\right)}{\mu_{\beta_0}\left(\beta^{-1}(\beta_0)\cap U\right)} \, d(\beta_\ast\mu_U)(\beta_0)\\
    &=\int_{\Delta_\mathcal{B}^\mathrm{irr}\cap \Delta_\mathcal{B}'\cap A_\mathcal{B}} \frac{\lambda\left(\beta^{-1}(\beta_0)\cap \delta^{-1}(A_\mathcal{D})\cap U\right)}{\lambda\left(\beta^{-1}(\beta_0)\cap U\right)} \, d(\beta_\ast\mu_U)(\beta_0),\\
    \end{align*}
    where we used~\eqref{eq:conditional-measure-equals-Lebesgue} in the third equality. Note that for every $\beta_0 \in V_{\mathcal{B}}$, the restriction of $(\beta,\delta)$ to the fiber over $\{\beta_0\} \times V_{\mathcal{D}}$ induces a diffeomorphism from $\beta^{-1}(\beta_0)\cap U$ onto $V_{\mathcal{D}}$. The preimage of $A_{\mathcal{D}}$ under this map is $\beta^{-1}(\beta_0)\cap \delta^{-1}(A_{\mathcal{D}})\cap U$. Since diffeomorphisms preserve Lebesgue null sets and $\lambda_{\mathcal{D}}(A_{\mathcal{D}})=0$, it follows that $\lambda\bigl(\beta^{-1}(\beta_0)\cap \delta^{-1}(A_{\mathcal{D}})\cap U\bigr)=0$, and hence $(\beta,\delta)_\ast\mu_u(A_\mathcal{B}\times A_\mathcal{D})=0$. 
    
    Repeating the same argument by switching the roles of $\mathcal{B}$ and $\mathcal{D}$, we reach the same conclusion under the assumption that $\lambda_\mathcal{B}(A_\mathcal{B})=0$ instead of $\lambda_\mathcal{D}(A_\mathcal{D})=0$. This proves that $(\beta,\delta)_\ast\mu_U\ll \lambda_\mathcal{B}\times\lambda_\mathcal{D}$ .
\end{proofclaim}

Combining Claim~\ref{claim:pushforward-localization-absolutely-continuous-wrt-Lebesgue} with~\eqref{eq:Goldman-equivalent-Lebesgue}, we conclude that $(\beta,\delta)_\ast\mu_U\ll (\beta,\delta)_\ast(\nuG)_U$ and thus $\mu_U\ll (\nuG)_U$ because $(\beta,\delta)$ is a diffeomorphism of $U$ onto $V_\mathcal{B}\times V_\mathcal{D}$.
\end{proof}

\appendix
\setcounter{section}{0}
\renewcommand\sectionname{Appendix}
\addcontentsline{toc}{section}{Appendix}
\renewcommand{\thesection}{\Alph{section}}

\section{Trigonometric computations}\label{apx:trig-computations}
\subsection{Trigonometric formulas}
We recall the essential formulas that govern the trigonometry of a hyperbolic triangle of side lengths $(a,b,c)$ and interior angles $(\alpha, \beta, \gamma)$. 
\subsubsection*{Hyperbolic laws of cosines}
\begin{align}
    \cos(\gamma)&=\frac{\cosh(a)\cosh(b)-\cosh(c)}{\sinh(a)\sinh(b)} \tag{C.1}\label{eq:hyperbolic-law-of-cosines-lengths},\\
    \cosh(c)&= \frac{\cos(\alpha)\cos(\beta)+\cos(\gamma)}{\sin(\alpha)\sin(\beta)}. \tag{C.2}\label{eq:hyperbolic-law-of-cosines-angles}
\end{align}
\subsubsection*{Hyperbolic law of sines}
\begin{equation}
    \sin(\alpha)=\sinh(a)\frac{\sin(\beta)}{\sinh(b)} \tag{S}\label{eq:hyperbolic-law-of-sines}.
\end{equation}
\subsubsection*{Four-parts formula}
\begin{equation}
    \cos (\gamma)\cosh(a)=\sinh(a)\coth(b)-\sin(\gamma)\cot(\beta). \tag{F}\label{eq:four-part-formula}
\end{equation}

\subsection{Calculations from Lemma~\ref{lem:induction-step-fibers-zero-measure}}
We present the explicit trigonometric computations omitted earlier in the proof of Lemma~\ref{lem:induction-step-fibers-zero-measure}. The reader may find it helpful to refer to Figure~\ref{fig:triangle-chain-B-and-Y-bis} when following the computations.

We will denote by $\varepsilon$ the angle $\angle(Y_1,C_j,B_{j-1})$. Since $\angle(C_i,C_j,Y_1)=\angle (B_{j-2}, C_j, B_{j-1})=\pi-\alpha_j/2$ by construction, we conclude that 
\[
\angle (C_i,C_j,B_{i-1})=\angle (Y_1,C_j,B_{j-1})=\varepsilon.
\]

\begin{figure}[h]
\centering
\begin{tikzpicture}[font=\sffamily, scale=1.1]
\node[anchor=south west,inner sep=0] at (0,0) {\includegraphics[width=10.8cm]{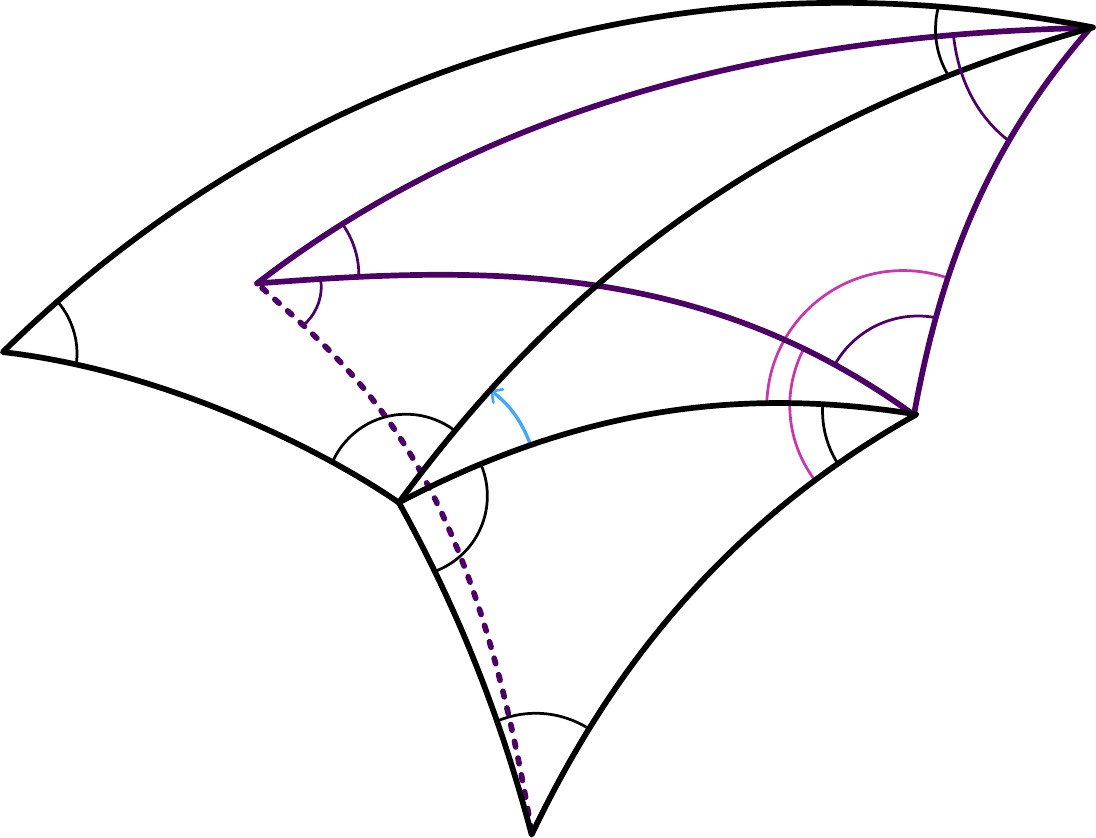}};
\begin{scope}
\fill (9.8,7.25) circle (0.07) node[right]{$C_i$};
\fill (8.2,3.78) circle (0.08) node[right]{$C_j$};
\draw (8.5, 7.7) node{\small $\pi-\alpha_i/2$};
\draw (8.25, 3.3) node{\small $\pi-\alpha_j/2$};
\draw (6.1, 0.85) node{\small $\pi-\beta_{j-1}/2$};
\draw (4.85, 2.8) node{\small $\beta_{j-2}/2$};
\draw (2.2, 3.2) node{\small $\pi-\beta_{i-1}/2$};
\draw (1.25, 4.55) node{\small $\beta_{i-2}/2$};
\draw (5.85, 5.7) node{$d_1$};
\draw (5.85, 4.05) node{$d_2$};
\draw (6.5, 2.17) node{$d_3$};
\end{scope}

\begin{scope}[apricot]
\fill (0.06,4.36) circle (0.07) node[left]{$B_{i-2}$};
\fill (3.6,3) circle (0.07) node[below left]{$B_{i-1}=B_{j-2}$};
\fill (4.77,0.04) circle (0.07) node[below]{$B_{j-1}$};
\end{scope}

\begin{scope}[sky]
\draw (4.9,3.9) node{$\gamma_{\ell_0}$};
\end{scope}

\begin{scope}[plum]
\fill (2.32,4.97) circle (0.07) node[left]{$Y_1$};   
\draw (2.8, 4.75) node[right]{$\eta$};
\draw (9.75, 6.25) node{\small $\pi-\alpha_i/2$};
\draw (9.1, 4.6) node{\small $\pi-\alpha_j/2$};
\draw (3.85, 5.35) node{\small $\pi-\upsilon_1/2$};
\end{scope}

\begin{scope}[mauve]
\draw (6.95, 3.6) node{$\varepsilon$};
\draw (7.5, 5.15) node{$\varepsilon$};  
\end{scope}
\end{tikzpicture}
\caption{The relevant configuration for the computations of Appendix~\ref{apx:trig-computations}.}
\label{fig:triangle-chain-B-and-Y-bis}
\end{figure}
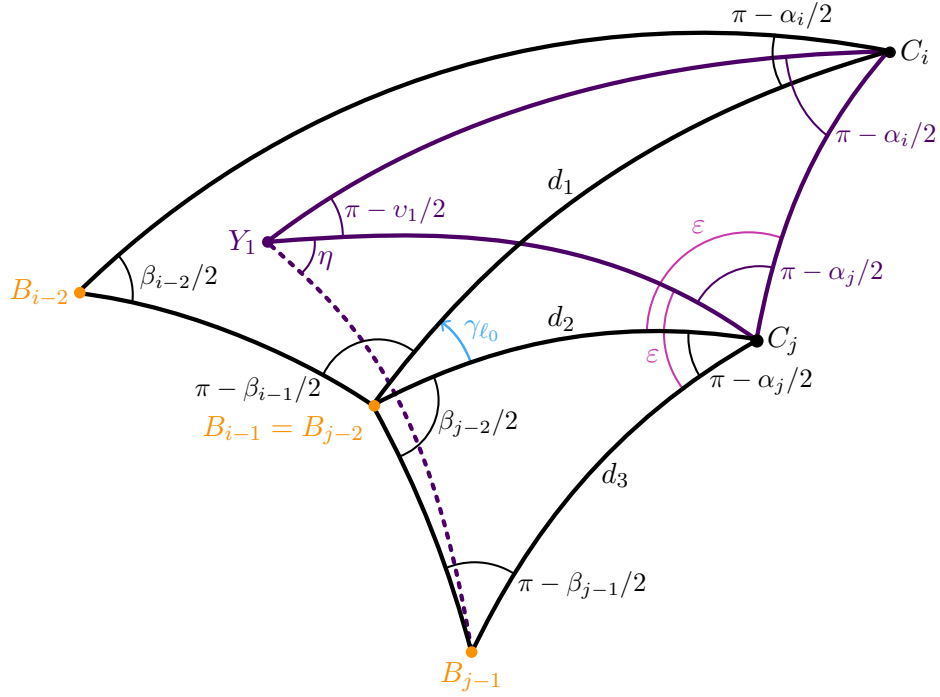

Let us now apply the four-parts formula~\eqref{eq:four-part-formula} to the triangle $(Y_1,C_j,B_{j-1})$ to obtain our main relation
\begin{equation}\label{eq:main-relation}
\cos(\varepsilon)\cosh\big(d(Y_1,C_j)\big)=\sinh\big(d(Y_1,C_j)\big)\coth\big(d(C_j,B_{j-1})\big)-\sin(\varepsilon)\cot(\eta).
\end{equation}
Our goal now is to express every term in~\eqref{eq:main-relation} using only $\cos(\upsilon_1/2)$ and the angles $\alpha_1,\ldots,\alpha_n$ and $\beta_1,\ldots,\beta_{n-3}$. To simplify the notation, we will write
\[
\left\{\begin{array}{l}
d_1=d(C_i,B_{i-1}),\\
d_2=d(C_j, B_{j-2}),\\
d_3=d(C_j, B_{j-1}).
\end{array}
\right.
\]
The three distances $d_1, d_2, d_3$ can be expressed purely in terms of $\alpha_1,\ldots,\alpha_n$ and $\beta_1,\ldots,\beta_{n-3}$.

We start with $\cos(\varepsilon)$ and $\sin(\varepsilon)$. The hyperbolic law of cosines~\eqref{eq:hyperbolic-law-of-cosines-lengths} applied to the triangle $(C_i,C_j,B_{j-2})$ gives
\begin{equation}\label{eq:cos(epsilon)}
\cos(\varepsilon)=\frac{\cosh\big(d(C_i,C_j)\big)\cosh(d_2)-\cosh(d_1)}{\sinh\big(d(C_i,C_j)\big)\sinh(d_2)}.
\end{equation}
By squaring~\eqref{eq:cos(epsilon)} and using $\sin(\varepsilon)^2=1-\cos(\varepsilon)^2$, as well as the identity $\sinh(x)^2\sinh(y)^2-\cosh(x)^2\cosh(y)^2=1-\cosh(x)^2-\cosh(y)^2$, we obtain
\begin{equation}\label{eq:sin(epsilon)^2}
\scalebox{1.1}{%
 $\sin(\varepsilon)^2=
 \frac{1-\cosh(d(C_i,C_j))^2-\cosh(d_2)^2-\cosh(d_1)^2+2\cosh(d(C_i,C_j))\cosh(d_2)\cosh(d_1)}{\sinh(d(C_i,C_j))^2\sinh(d_2)^2}.$%
 }
\end{equation}
Both equations~\eqref{eq:cos(epsilon)} and~\eqref{eq:sin(epsilon)^2} involve the distance $d(C_i,C_j)$ which we compute by applying the hyperbolic law of cosines~\eqref{eq:hyperbolic-law-of-cosines-angles} to the triangle $(C_i,C_j,Y_1)$, we obtain
\begin{equation}\label{eq:cosh(C_i,C_j)}
\cosh\big(d(C_i,C_j)\big)=\frac{-\cos\left(\frac{\upsilon_1}{2}\right)+\cos\left(\frac{\alpha_i}{2}\right)\cos\left(\frac{\alpha_j}{2}\right)}{\sin\left(\frac{\alpha_i}{2}\right)\sin\left(\frac{\alpha_j}{2}\right)}.
\end{equation}

Finally, we compute the term $\cosh\big(d(Y_1,C_j)\big)$ appearing in~\eqref{eq:main-relation} by applying the hyperbolic law of cosines~\eqref{eq:hyperbolic-law-of-cosines-angles} to the triangle $(C_i,C_j,Y_1)$ and get
\begin{equation}\label{eq:cosh(Y_1C_j)}
\cosh\big(d(Y_1,C_j)\big)=\frac{-\cos\left(\frac{\alpha_i}{2}\right)+\cos\left(\frac{\alpha_j}{2}\right)\cos\left(\frac{\upsilon_1}{2}\right)}{\sin\left(\frac{\alpha_j}{2}\right)\sin\left(\frac{\upsilon_1}{2}\right)}.
\end{equation}
On the other hand, the hyperbolic law of sines~\eqref{eq:hyperbolic-law-of-sines} in the triangle $(C_i,C_j,Y_1)$ gives
\begin{equation}\label{eq:sinh(Y_1C_j)}
\sinh\big(d(Y_1,C_j)\big)=\sinh\big(d(C_i,C_j)\big)\frac{\sin\left(\frac{\alpha_i}{2}\right)}{\sin\left(\frac{\upsilon_1}{2}\right)}.    
\end{equation}

We now rewrite~\eqref{eq:main-relation} by plugging in the values of $\cos(\varepsilon)$, $\cosh\big(d(Y_1,C_j)\big)$, $\sinh\big(d(Y_1,C_j)\big)$ from~\eqref{eq:cos(epsilon)},~\eqref{eq:cosh(Y_1C_j)},~\eqref{eq:sinh(Y_1C_j)}. After clearing out the denominators, we obtain
\begin{equation}\label{eq:first-plug-in}
\scalebox{1}{%
$\begin{split}
\sin&(\varepsilon)\sinh\big(d(C_i,C_j)\big)\sinh(d_2)\cot(\eta)\sin\left(\frac{\upsilon_1}{2}\right)\sin\left(\frac{\alpha_j}{2}\right) =\\
&\sinh\big(d(C_i,C_j)\big)^2\sin\left(\frac{\alpha_i}{2}\right)\sin\left(\frac{\alpha_j}{2}\right)\coth(d_3)\sinh(d_2)\\
&+\left(\cosh\big(d(C_i,C_j)\big)\cosh(d_2)-\cosh(d_1)\right)\left(\cos\left(\frac{\alpha_i}{2}\right)-\cos\left(\frac{\alpha_j}{2}\right)\cos\left(\frac{\upsilon_1}{2}\right)\right).
\end{split}$%
}
\end{equation}
In order to transform~\eqref{eq:first-plug-in} into a polynomial in $\cos(\upsilon_1/2)$, we start by squaring both sides. We then use the trigonometric relation 
\[
\cot(\eta)^2=\frac{\cos(\eta)^2}{1-\cos(\eta)^2}
\]
and multiply both sides of the squared equation by $1-\cos(\eta)^2$. We continue by rewriting both $\cos(\eta)^2$ as a degree $2m$ Chebyshev polynomial in $\cos(\upsilon_1/2)$ with leading coefficient $2^{2m-2}$. We also plug in the expression~\eqref{eq:sin(epsilon)^2} for $\sin(\varepsilon)^2$. Now, we rewrite the term $\sin(\upsilon_1/2)^2$ as $1-\cos(\upsilon_1/2)^2$ and each $\sinh\big(d(C_i,C_j)\big)^2$ as $\cosh\big(d(C_i,C_j)\big)^2-1$. Finally, we replace each $\cosh\big(d(C_i,C_j)\big)$ by the expression~\eqref{eq:cosh(C_i,C_j)}. These algebraic manipulations turn~\eqref{eq:first-plug-in} into a degree $2m+4$ polynomial equation in $\cos(\upsilon_1/2)$ (the angles $\alpha_i,\alpha_j$ and the distances $d_1,d_2,d_3$ are treated as constants). We will not write this polynomial explicitly, but we will compute its leading coefficient and show it is never zero. If we track all the highest degree terms in $\cos(\upsilon_1/2)$ for each term appearing in~\eqref{eq:first-plug-in} squared, we obtain the following. We use the symbol $\approx$ to mean that we extracted the highest degree term in $\cos(\upsilon_1/2)$. First, we study the left hand side of~\eqref{eq:first-plug-in} squared and multiplied by $1-\cos(\eta)^2$:
\begin{align*}
    &\sin(\varepsilon)^2\sinh\big(d(C_i,C_j)\big)^2\sinh(d_2)^2\approx -\cosh\big(d(C_i,C_j)\big)^2\approx \frac{-\cos\left(\frac{\upsilon_1}{2}\right)^2}{\sin\left(\frac{\alpha_i}{2}\right)^2\sin\left(\frac{\alpha_j}{2}\right)^2}\\
    &\cos(\eta)^2\approx 2^{2m-2}\cos\left(\frac{\upsilon_1}{2}\right)^{2m}\\
    &\sin\left(\frac{\upsilon_1}{2}\right)^2\approx -\cos\left(\frac{\upsilon_1}{2}\right)^2.
\end{align*}
So the highest degree term in $\cos(\upsilon_1/2)$ appearing on the left hand side after squaring~\eqref{eq:first-plug-in} and multiplying it by $1-\cos(\eta)^2$ is
\[
\frac{2^{2m-2}}{\sin\left(\frac{\alpha_i}{2}\right)^2}\cdot \cos\left(\frac{\upsilon_1}{2}\right)^{2m+4}.
\]
Next, for the right hand side:
\begin{align*}
    &1-\cos(\eta)^2\approx -2^{2m-2}\cos\left(\frac{\upsilon_1}{2}\right)^{2m}\\
    &\sinh\big(d(C_i,C_j)\big)^2\approx \cosh\big(d(C_i,C_j)\big)^2\approx \frac{\cos\left(\frac{\upsilon_1}{2}\right)^2}{\sin\left(\frac{\alpha_i}{2}\right)^2\sin\left(\frac{\alpha_j}{2}\right)^2}\\
    &\cosh\big(d(C_i,C_j)\big)\approx \frac{-\cos\left(\frac{\upsilon_1}{2}\right)}{\sin\left(\frac{\alpha_i}{2}\right)\sin\left(\frac{\alpha_j}{2}\right)}.
\end{align*}
This gives the following highest degree term on the right hand side
\[
-2^{2m-2}\frac{\left(\coth(d_3)\sinh(d_2)+\cosh(d_2)\cos\left(\frac{\alpha_j}{2}\right)\right)^2}{\sin\left(\frac{\alpha_i}{2}\right)^2\sin\left(\frac{\alpha_j}{2}\right)^2}\cdot \cos\left(\frac{\upsilon_1}{2}\right)^{2m+4}.
\]
If we bring back all the terms to the left hand side, the leading coefficient will be
\[
\frac{2^{2m-2}}{\sin\left(\frac{\alpha_i}{2}\right)^2}\left(1+\frac{\left(\coth(d_3)\sinh(d_2)+\cosh(d_2)\cos\left(\frac{\alpha_j}{2}\right)\right)^2}{\sin\left(\frac{\alpha_j}{2}\right)^2}\right)>0.
\]
This shows that the resulting polynomial in $\cos(\upsilon_1/2)$ is indeed of degree $2m+4$.

\bibliographystyle{amsalpha}
\bibliography{references.bib}
\end{document}